\documentclass[11pt]{amsart}

\usepackage[english]{babel}
\usepackage{enumitem}
\usepackage[ansinew]{inputenc}
\usepackage{amsmath,amsthm}
\usepackage{hyperref}
\usepackage[T1]{fontenc}
\usepackage{lmodern}
\usepackage{amssymb}
\usepackage{amscd}
\usepackage{marginnote}
\usepackage[all]{xy} 
\usepackage{stmaryrd} 
\usepackage{tikz-cd}
\usepackage{cleveref}
\setenumerate[1]{label=(\alph*)}
\setenumerate[2]{label=(\roman*)}

\theoremstyle{plain}
\newtheorem{thm}{Theorem}
\numberwithin{thm}{subsection}
\newtheorem{lem}[thm]{Lemma}
\newtheorem{cor}[thm]{Corollary}
\newtheorem{prop}[thm]{Proposition}
\newtheorem{definition}[thm]{Definition}

\newtheoremstyle{shadow}{}{}{\itshape}{}{\bfseries}{.}{.5em}{\thmname{#3}}
\theoremstyle{shadow}

\theoremstyle{definition}
\newtheorem{defin}[thm]{Definition}

\newtheorem{exmpl}[thm]{Example}
\theoremstyle{remark}
\newtheorem{rem}[thm]{Remark}

\newcommand{\meas}{\mathrm{Meas}}

\newcommand{\Fp}{\mathbb{F}_p}

\newcommand{\Zp}{\mathbb{Z}_p}
\newcommand{\G}{{\mathbb{Z}^{*}_p/\{\pm 1\}}}
\newcommand{\Qp}{\mathbb{Q}_p}

\newcommand{\ZZ}{\mathbb{Z}}
\newcommand{\QQ}{\mathbb{Q}}
\newcommand{\CC}{\mathbb{C}}
\newcommand{\NN}{\mathbb{N}}
\newcommand{\NNo}{{\mathbb{N}_0}}
\newcommand{\Cev}{\mathrm{C}^{\mathrm{even}}}
\newcommand{\Cont}{\mathrm{C}}
\newcommand{\MString}{\mathrm{MString}}
\newcommand{\MSpin}{\mathrm{MSpin}}
\newcommand{\KO}{\mathrm{KO}}
\newcommand{\tmf}{\mathrm{tmf}}
\newcommand{\MF}{\mathrm{MF}}
\newcommand{\MFp}{\mathrm{MF}_p}

\newcommand{\Momeulzwei}{\mathrm{Mom}^\mathrm{Euler}_{\geq 4}}
\newcommand{\Momeuleins}{\mathrm{Mom}^\mathrm{Euler}_{\geq 2}}
\newcommand{\Momzero}{\mathrm{Mom}^{(0)}_{\geq 2m}}
\newcommand{\Momzerozwei}{\mathrm{Mom}^{(0)}_{\geq 4}}
\newcommand{\Momeul}{\mathrm{Mom}^\mathrm{Euler}_{\geq 2m}}
\newcommand{\coker}{\mathrm{coker}}
\newcommand{\id}{\mathrm{id}}
\newcommand{\im}{\mathrm{im}}
\newcommand{\Bp}{\ensuremath{\mathrm{(B)}_p}}
\newcommand{\Bpa}{\ensuremath{\widetilde{(\mathrm{B})_p}}}

\title[Simultaneous Kummer congruences and $\mathbb{E}_\infty$-orientations]{Simultaneous Kummer congruences and $\mathbb{E}_\infty$-orientations of KO and tmf\\ }

\author{Niko Naumann}
\address{Fakult\"at f\"ur Mathematik Universit\"at Regensburg  \\ 93040 Regensburg }
\email{niko.naumann@mathematik.uni-regensburg.de }
\thanks{}
\date{}

\author{Johannes Sprang}
\address{Fakult\"at f\"ur Mathematik Universit\"at Regensburg  \\ 93040 Regensburg }
\email{johannes.sprang@mathematik.uni-regensburg.de }
\thanks{This work was supported by the Collaborative Research Centre SFB1085, funded by the DFG}
\date{}

\hypersetup{
pdfauthor = {J. Sprang},
pdftitle = {P-adic interpolation and multiplicative orientations}
}


\begin{document}

\begin{abstract}

Building on results of M.~Ando, M.J.~Hopkins and C.~Rezk, we show the existence of uncountably many $\mathbb{E}_\infty$-String orientations of real K-theory KO and of topological modular forms tmf, generalizing the $\hat{A}$- (resp.~the Witten) genus. Furthermore, the obstruction to lifting an $\mathbb{E}_\infty$-String orientations from KO to tmf is identified with a classical Iwasawa-theoretic condition.\\
The common key to all these results is a precise understanding of the classical Kummer congruences, imposed {\em for all primes simultaneously}. This result is of independent arithmetic interest.

\end{abstract}

\maketitle
\tableofcontents

\section[Introduction]{Introduction}

Orientations of bundles with respect to cohomology theories play an important role in both topology and geometry. One example of such an orientation is given by the $\hat{A}$-genus defined by Atiyah, Bott and Shapiro \cite{atiyah_bott_shapiro}. This orientation is geometrically important since it assigns interesting invariants to spin manifolds and is related to the index of a Dirac-operator through the Atiyah--Singer index theorem. Part of the relevance of this orientation for homotopy theory stems from the fact that the unit of the real K-theory spectrum $\KO$ factors through the Atiyah-Bott-Shapiro orientation. By  \cite{joachim}, the Atiyah--Bott--Shapiro orientation refines to an $\mathbb{E}_\infty$-map 
\[
	\MSpin \rightarrow \KO
\]
on the Spin-bordism spectrum $\MSpin$.
The Witten genus suggested a similar factorization of the unit of $\tmf$ through the spectrum $\MString$ of String-bordism. In his ICM-talk \cite{hopkins_ICM}, M.J.~Hopkins announced the existence of an $\mathbb{E}_\infty$-map
\[
	\MString\rightarrow \tmf
\]
providing the desired factorization and giving the Witten genus on homotopy groups. More generally, he announced a completely arithmetic description of the set of all 
(homotopy classes of) $\mathbb{E}_\infty$-String orientations of $\KO$ and $\tmf$. The details appeared in the joint work \cite{andohopkinsrezk} of M. Ando, M.J. Hopkins and C. Rezk. They assign to every $\mathbb{E}_\infty$-String orientation a characteristic series, thus defining a map
\[
  b\colon \pi_0 \mathbb{E}_\infty (\MString,X)\rightarrow \prod_{k\geq 4} \pi_{2k}(gl_1 (X)) \otimes \QQ,\quad \text{$X=\KO$ or $\tmf$}.
\]
They also describe the image of this map in a purely arithmetic way as consisting of sequences of rational numbers (respectively rational modular forms) simultaneously satisfying a specific $p$-adic interpolation condition at every prime $p$. A similar result holds for $\mathbb{E}_\infty$-Spin orientations of $\KO$. Thus, the problem of understanding {\em all} $\mathbb{E}_\infty$-orientations was reduced to a purely arithmetic problem, but this problem remained open. The only sequences well-known to satisfy the requisite properties are those corresponding to the Atiyah--Bott--Shapiro orientation, respectively the Witten genus.\par
Our main result is a description of all $\mathbb{E}_\infty$-orientations. In \Cref{thm:many_orient}, we will for example show the existence of bijections
\[
\begin{tikzcd}[row sep=tiny]
	 \ZZ^\NNo\arrow{r}{\cong} & \pi_0 \mathbb{E}_\infty (\MSpin,\KO)\\
	 \ZZ^\NNo\arrow{r}{\cong} & \pi_0 \mathbb{E}_\infty (\MString,\KO).
\end{tikzcd}
\]
The group structure of $\ZZ^\NNo$ will be compatible with the torsor action of $[bspin,gl_1(\KO)]$ on  $\mathbb{E}_\infty (\MSpin,\KO)$, and similarly for $\mathbb{E}_\infty$-String orientations. These bijections are effectively computable in the following sense: Given a sequence in $\ZZ^\NNo$, it is possible to calculate finitely many terms of the characteristic sequence, i.e.~the image under $b$, of the corresponding orientation, see \Cref{ex:moments}.\par
As a further illustration, we also determine in \Cref{thm:zeta_ideal} which $\mathbb{E}_\infty$-String orientations of KO lift through the evaluation at the cusp map tmf$\to$KO, and in \Cref{thm:stringspinko} we determine
which $\mathbb{E}_\infty$-String orientations of KO factor through the canonical map
$\MString\to\MSpin$.\par

The proof of \Cref{thm:many_orient} reduces to understanding the condition on a sequence of integers $b_2,b_4,b_6,\ldots\in\ZZ$ that for every prime $p$ there should exist a $p$-adic measure $\mu_p$ on $\G$ such that for all $k\ge 1$ we have $\int_\G x^{2k}\,d\mu_p(x)=
b_{2k}.$ While for a {\em fixed} prime $p$, this condition is equivalent to the familiar Kummer congruences
from number theory, prior to the present work it was even unknown if there exists 
a single such sequence which is not identically zero. In \Cref{cor2} and \Cref{rem1} we show that these sequences can be characterized by explicit, recursive congruences.\par
We conclude the introduction with a brief outline of the paper.\par
After some preliminaries on $p$-adic analysis, we construct in \cref{subset:onb} 
a basis for the $p$-adic Banach space of continuous maps $\G\to\ZZ_p$ which consists
of polynomials (\Cref{prop:onb}). This is similar to, but slightly more involved than, the classical example 
of the Mahler basis for maps $\ZZ_p\to\ZZ_p$, which consists of the polynomials
$X \choose n$ ($n \ge 0$).\\
In \cref{sec:simultaneous} we record the fairly straight forward deduction of \Cref{cor2}
from some simple observations about the shape of the polynomial bases constructed above. We also compute explicitly (\Cref{ex:moments}) the resulting parametrization of all
sequences $b_2,\ldots$ as above. This shows in particular that any such sequence which is not identically zero must grow quite rapidly. It remains an interesting problem to find natural 
additional conditions which force the zero sequence to be the only solution. Equivalently, 
this is asking for natural conditions on an $\mathbb{E}_\infty$-map $\MSpin\to\KO$
which force it to be the $\hat{A}$-genus.\par
Sections \ref{sec:string_of_ko} and \ref{sec:string_of_tmf} contain a refined discussion and 
a proof of \Cref{thm:many_orient} above.\par
The final sections \ref{sec:eval_at_cusp} and \ref{sec:comparing} are meant to illustrate the
range of applicability of our methods. In the first, we ask which $\mathbb{E}_\infty$-String orientations of $\KO$ lift through the evaluation of the cusp map $\tmf\to\KO$, and we 
find a non-trivial condition involving the zeta-ideal from number theory (\Cref{thm:zeta_ideal}).
In the second one, we ask which $\mathbb{E}_\infty$-String orientations of $\KO$
factor though the canonical map $\MString\to\MSpin$. Since this map is a $K(1)$-local
equivalence for all $p$ and the $p$-completion of $\KO$ is $K(1)$-local, it might seem surprising that in fact most maps do not such factor (\Cref{thm:stringspinko}).

\section*[Acknowledgment]{Acknowledgment}
The unpublished PhD thesis \cite{nerf_phd} of Christian Nerf contains preliminary
observations on the problems addressed here.
The second author would like to thank Christian Nerf. It was his PhD thesis which drew his attention to the set of problems studied in this paper. He is also grateful for remarks and comments by Uli Bunke, Thomas Nikolaus and Michael V\"olkl. Last but not least, he would like to thank Guido Kings for introducing him to the beautiful subject of $p$-adic interpolation. The first author thanks Charles Rezk for suggesting \Cref{thm:stringspinko}. We also thank an anonymous referee for helpful comments improving the readability.

\section{The arithmetic of simultaneous Kummer congruences}

\subsection[Polynomial Banach bases]{Polynomial Banach bases}

Fix a prime $p$. The aim of this section is to give a general principle for the construction of an orthonormal basis of the $\Qp$-Banach spaces of continuous maps
\[
	\Cont(\Zp^*,\Qp)\text{ and }\Cev(\Zp^*,\Qp):=\{f\in\Cont(\Zp^*,\Qp) \mid f(-x)=f(x)\} .
\]

As an application, we construct an explicit orthonormal basis consisting of polynomial functions. These will serve as a substitute of the familiar Mahler basis ${  \binom{X}{n}}_{n\geq 0}$ of $\Cont(\Zp,\Qp)$ and will lead to explicit congruences characterizing 
the sequences of moment of measures on $\Zp^*/ \{ \pm 1\}$, see \Cref{cor2}.

We start with some recollections on $p$-adic functional analysis, as
initiated by J-P. Serre \cite{serre1}. Let $|\cdot|$ denote the valuation of $\Qp$ normalized by $|p|=1/p$.

\begin{defin}
	A $\Qp$-Banach space $(E,\Vert\cdot\Vert)$ is a vector space $E$ over $\Qp$, which is complete with respect to an {\em ultrametric} norm $\Vert\cdot\Vert$, i.e. the norm satisfies
	\[
		\Vert v+w \Vert\leq \max (\Vert w\Vert,\Vert v\Vert).
	\]
	We also require that the value group of $\Vert\cdot\Vert$ be contained in that of $|\cdot|$.
\end{defin}

Our main examples in the following will be $(\Cev(\Zp^*,\Qp),\Vert \cdot \Vert_{\mathrm{sup}})$ and $(\Cont(\Zp^*,\Qp),\Vert \cdot \Vert_{\mathrm{sup}})$, the space of continuous functions with the supremum norm.

\begin{defin}
	A sequence $(v_k)_{k\geq 0}$ in a $\Qp$-Banach space $(E,\Vert\cdot\Vert)$ is called an \emph{orthonormal basis} if each $x\in E$ has a unique representation as
	\[
		x=\sum_{k\geq 0} c_k v_k
	\]
	with $c_k\in \Qp$,  $|c_k|\rightarrow 0$ as $k\rightarrow \infty$ and
	\[
		\Vert x \Vert = \max_{k\geq 0} |c_k|.
	\]
\end{defin}

The following result will be useful.

\begin{lem}[{\cite[Lemme 1]{serre1}}]\label{lem:lem1}
Let $(E,\Vert\cdot\Vert)$ be a $\Qp$-Banach space and set $E_0:=\{x\in E: \Vert x\Vert \leq 1\}$ and $\overline{E}:=E_0/p E_0$. A sequence $(v_n)_{n\geq 0}$ is a orthonormal basis for $E$ if and only if $v_n\in E_0$ for all $n\geq 0$, and the reductions $\overline{v_n}\in\overline{E}$ form an $\Fp$-basis $\{\overline{v_n} \}_{n\ge 0}$ for $\bar{E}$.
\end{lem}

We also recall the notion of a $p$-adic measure.

\begin{defin}
	Let $R$ be a commutative ring which is complete and separated in its $p$-adic topology, i.e.
	the canonical map $R\to\lim_n R/p^nR$ is an isomorphism. Given a pro-finite abelian group $G$, {\em an $R$-valued measure} $\mu$ on $G$ is an $R$-linear map:
	\[
		\begin{tikzcd}
			\mu\colon \Cont(G,R) \rightarrow R.
		\end{tikzcd}
	\]
	Note that such a map is automatically continuous because one easily checks that the compact-open topology on $\Cont(G,R)$ is the $p$-adic one. For given $f\in \Cont(G,R)$, we will often use the following notation:
	\[
		\int_{G}{fd\mu}:=\mu(f).
	\]
	The set of all $R$-valued measures on $G$ will be denoted $\meas(G,R)$. Convolution of measures
	\[
		(\mu,\nu)\mapsto \mu*\nu
	\]
	defined by
		\[
		(\mu*\nu)(f):=\int_{G}\int_{G}f(xy)d\mu(x)d\nu(y)
	\]
	gives $\meas(G,R)$ the structure of an $R$-algebra.
Finally, recall that given a measure $\mu$, for every $n\in\mathbb{Z}$ the $n$-th moment of $\mu$ is defined to be $\int_{G}{ x^n\,d\mu(x)}\in R$.
\end{defin}

\begin{rem}
	In our applications, the pro-finite group $G$ will be either $G=\Zp^*$ or $G=\G$ and the coefficient ring $R$ will be either $\Zp$ or the ring of $p$-adic modular forms $\MFp$.
\end{rem}

\subsection[The $p$-adic digit principle]{A variant of the $p$-adic digit principle}
Recall the fixed prime $p$. Our construction here is an adaptation of what K. Conrad calls \emph{digit principle} in \cite{conrad1}.\par
For the formulation we need the following definition:
\begin{defin}
	Each integer $n\geq 0$ can be uniquely written in the form
	\[
		n=a_0 + \sum_{i\geq 1} a_i \varphi(p^i) \quad a_0\in\{0,...,p-2\}, a_i\in\{0,...,p-1\, \} \,\,(i\ge 1),
	\]
	where $\varphi$ denotes Euler's totient function. We call this expansion the {\em $\varphi$-adic expansion of $n$ to the base $p$}.
\end{defin}
\begin{proof}
	Recall that $\varphi(p^i)=(p-1)p^{i-1}$ $(i\ge 1)$. The claim now follows immediately by choosing the unique $a_0\in\{0,...,p-2\}$ such that $(p-1)$ divides $(n-a_0)$ and then using the $p$-adic expansion of $\frac{n-a_0}{(p-1)}$.
\end{proof}

\begin{exmpl}
		In case $p=2$ we always have $a_0=0$, and the $\varphi$-adic expansion of $n=\sum_{i\geq 1}a_i 2^{i-1}$ to the base $2$ is simply the $2$-adic expansion with index shifted by $1$.
\end{exmpl}

Now assume we are given a family of functions $(e_j)_{j\geq 0}$ with $e_j\in\Cont(\Zp^*,\Qp)$. We define a new sequence $(E_n)_{n\geq0}$ by
\[
	E_n:=\prod_{j\geq 0} e_j^{a_j}\in\Cont(\Zp^*,\Qp),
\]
where the $a_j$ are the digits in the $\varphi$-adic expansion of $n=a_0 + \sum_{j\geq 1} a_j \varphi(p^j)$. We call $(E_n)_{n\ge 0}$ the {\em extension of $(e_j)_{j\ge 0}$ by $\varphi$-digits}. 
\begin{exmpl}
We have $E_{\varphi(p^j)}=e_j$ for $j\geq 1$, $E_0=1$ and
		\[
				E_1=
				\begin{cases}
					e_0 & p\neq 2\\
					e_1 & p=2.
				\end{cases}
		\]
\end{exmpl}

The following result will be referred to as the \emph{$\varphi$-adic digit principle}. It is an adaptation of \cite[Theorem 3]{conrad1}.

\begin{thm}\label{thm1}
	Let $(e_j)_{j\geq 0}$, $e_j\in\Cont(\Zp^*,\Qp)$ be a sequence of functions satisfying the following three conditions:
	\begin{enumerate}
		\item For all $j\geq0$: $e_j(\Zp^*)\subseteq \Zp$ and $e_0(\Zp^*)\subseteq \Zp^*$.
		\item For all $n\geq 1$ the reductions $(\bar{e}_{j})_{0\leq j\leq n-1}\in \Cont(\Zp^*,\Fp)$ are constant on all cosets of $1+p^n\Zp\subseteq\Zp^*$.
		\item For all $n\ge 1$, the map
		\[
		\begin{tikzcd}[row sep=tiny]
			\Zp^*/(1+p^n\Zp) \arrow{r} & \Fp^*\times \Fp^{n-1} \\
			x \arrow[mapsto]{r} & (\bar{e}_{0}(x),(\bar{e}_{j}(x))_{1\leq j\leq n-1})
		\end{tikzcd}
		\]
		is bijective.
	\end{enumerate}
	Then the sequence $(E_n)_{n\geq 0}$ obtained by extension of $(e_j)_{j\geq 0}$ by $\varphi$-digits is an orthonormal basis for $\Cont(\Zp^*,\Qp)$. If, moreover, the functions $E_{2n}$	are even for $n\geq 0$ then $(E_{2n})_{n\geq 0}$ is an orthonormal basis for $\Cev(\Zp^*,\Qp)$.
\end{thm}
\begin{proof}
We check the conditions of \Cref{lem:lem1} for $v_n:=E_n$ ($n\ge 0$).
Assumption $(a)$ and the construction of the $E_n$ from the $e_n$ implies 
 \[ E_n\in\Cont(\Zp^*,\Zp)=\{ f\in\Cont(\Zp^*,\Qp)\, \mid \, \Vert f\Vert\leq 1\}\]
 for all $n\ge 0$. We now check that the reductions  
  \[
	(\bar{E}_{n})_{n\geq 0}\in\Cont(\Zp^*,\Fp)
	\]
	of the $E_n$ are an $\Fp$-basis of $\Cont(\Zp^*,\Fp)$. 
	For every $n\ge 1$, assumption $(b)$ and the construction of the $E_m$ ($0\leq m\leq\varphi(p^n)-1$) from the $e_k$ ($0\leq k\leq n$) shows that these $E_m$ are constant on all cosets of $1+p^n\Zp\subseteq\Zp^*$. 
Since every continuous map $\Zp^*\to\Fp$ is in fact locally constant, we have
\[
		\Cont(\Zp^*,\Fp)=\varinjlim_n \mathrm{Maps}(\Zp^*/(1+p^n\Zp),\Fp),
	\]
	and it suffices to show that for every $n\ge 1$, the 
	\[
		(\bar{E}_{m})_{0\leq m\leq \varphi(p^n)-1} \in \mathrm{Maps}(\Zp^*/(1+p^n\Zp),\Fp)
	\]
	are an $\Fp$-basis of $\mathrm{Maps}(\Zp^*/(1+p^n\Zp),\Fp)$. By counting dimensions we are further reduced to showing that the family $(\bar{E}_{m})_{0\leq m\leq \varphi(p^n)-1}$ spans $\mathrm{Maps}(\Zp^*/(1+p^n\Zp),\Fp)$. To see this, define functions $h_v\in \mathrm{Maps}(\Zp^*/(1+p^n\Zp),\Fp)$ for $v\in \Zp^*/(1+p^n\Zp)$ by
	\[
		h_v(w):=\prod_{a\in \Fp^*\setminus \{1\}} \left( \frac{\bar{e}_0(w)}{\bar{e}_0(v)}-a \right)\prod_{j=1}^{n-1}\left(1-(\bar{e}_{j}(w)-\bar{e}_{j}(v))^{p-1}\right)
	\]
	with the usual convention that the value of an empty product is $1$. Expanding the product defining $h_v$ shows that $h_v$ is a linear combination of monomials in the $\bar{e}_j$. Since the exponent of $\bar{e}_0$ is always between $0$ and $p-2$ and all other exponents are at most $ p-1$, it is obvious that $h_v$ is a $\Fp$-linear combination of the $(\bar{E}_m)_{0\leq m\leq \varphi(p^n)-1}$. But the definition of $h_v$ and assumption $(c)$ shows
	\[
		h_v(w)=
		\begin{cases} 
		-1 & \mathrm{if}\quad v=w\\
		0 & \mathrm{else}
		\end{cases},
	\]
	so the $(\bar{E}_m)_{0\leq m\leq \varphi(p^n)-1}$ span $\mathrm{Maps}(\Zp^*/(1+p^n\Zp),\Fp)$, as desired.\par
	For the additional claim about $\Cev(\Zp^*,\Qp)$, note that for $n\geq 2$ the $\frac{\varphi(p^n)}{2}$ functions $(\bar{E}_{2m})_{0\leq 2m\leq \varphi(p^n)-2}$ are linearly independent in $\mathrm{Maps}(\Zp^*/(1+p^n\Zp),\Fp)$ by the above argument. But since they are contained in the $\frac{\varphi(p^n)}{2}$-dimensional subspace $\mathrm{Maps}^\mathrm{even}(\Zp^*/(1+p^n\Zp),\Fp)$ they form an $\Fp$-basis for $\mathrm{Maps}^\mathrm{even}(\Zp^*/(1+p^n\Zp),\Fp)$. So the claim follows by \Cref{lem:lem1} and the observation that 
	\[
		\Cev(\Zp^*,\Fp)=\varinjlim_n \mathrm{Maps}^\mathrm{even}(\Zp^*/(1+p^n\Zp),\Fp).
	\]
\end{proof}

\begin{rem}
	The canonical projection
	\[
		\begin{tikzcd}
			\Zp^* \ar[two heads]{r}{\mathrm{pr}} & \G
		\end{tikzcd}
	\]
	induces an isomorphism of $p$-adic Banach spaces
	\[
		\begin{tikzcd}
			\Cont(\G,\Qp) \ar{r}{\cong} & \Cev(\Zp^*,\Qp).
		\end{tikzcd}
	\]
	We will freely use this identification to interpret even functions on $\Zp^*$ as functions on $\G$, and conversely.
\end{rem}

\subsection{Construction of an orthonormal basis}\label{subset:onb}
Fix a prime number $p$. In this section, we will give explicit examples of orthonormal bases for $\Cev(\Zp^*,\Qp)$ and $\Cont(\Zp^*,\Qp)$ consisting of polynomials,
and we will use them to study the congruences characterizing the moments of measures on $\G$, i.e. the ``generalized Kummer congruences'' of \cite[Definition 9.6]{andohopkinsrezk}.

\begin{definition}\label{def:e_and_E}
For an integer $a\in\mathbb{Z}/p^j \mathbb{Z}$ let us write $\tilde{a}$ for the unique lift of $a$ in the set of representatives
\[
	\left\{-\left\lfloor\frac{p^j}{2}\right\rfloor, ..., p^j-1-\left\lfloor\frac{p^j}{2}\right\rfloor \right\}.
\]
For odd primes this set is symmetric with respect to zero.
\begin{itemize}
\item[i)]\label{it:e} For $j\geq 0$, define  polynomials $e_j^{(p)}\in\mathbb{Q}[X]$ by
\begin{equation*}
\begin{split}
	e^{(p)}_0&:=X,\\
	e^{(p)}_j&:=\frac{1}{(p^j)!}\prod_{a\in(\mathbb{Z}/p^j\mathbb{Z})^\times } (X-\tilde{a}).
\end{split}
\end{equation*}
\item[ii)]\label{it:E} Next, define by extension of $\varphi$-digits
\[
	E^{(p)}_n:=\prod_{j\geq 0} e_j^{a_j},\quad\mbox{ if } n=a_0 + \sum_{j\geq 1} a_j \varphi(p^j),\, n\geq 0.
\]
\end{itemize}
\end{definition}

Note that each $E^{(p)}_n$ is a polynomial of degree $n$ with coefficients in $\QQ$, and that $E^{(p)}_{n}(X)$ is even if $n$ is even.

 Before we prove that $(E^{(p)}_n)_{n\geq 0}$ is an orthonormal basis for $\Cont(\Zp^*,\Qp)$, we recall the following congruence due to E. Lucas:
\begin{lem}[\cite{lucas}]
	Let $p$ be a prime and $n,m\geq 0$ and $0\leq n_0,m_0 \leq p-1$ integers. Then the following congruence holds
	\[
		\binom{pn+n_0}{pm+m_0}\equiv \binom{n_0}{m_0}\binom{n}{m} \mod p
	\]
	with the usual convention that $\binom{a}{b}=0$ if $a<b$.
\end{lem}

\begin{cor}\label{cor3}
	Let $n\geq 0$ be an integer and $p$ a prime. If
	\[
		n=\sum_{j\geq 0} n_j p^j\quad 0\leq n_j \leq p-1
	\]
	is the $p$-adic expansion of $n$ then, for all $j\ge 0$, we have
	\[
		\binom{n}{p^j}\equiv n_j \mod p.
	\]
\end{cor}
\begin{proof}
	For integers $n,m\geq0$ with $p$-adic expansion $n=\sum_{j\geq 0}^{r} n_j p^j$ and $m=\sum_{j\geq 0}^{r} m_j p^j$ iterated application of Lucas' congruence gives
	\[
		\binom{n}{m}\equiv \binom{n_0}{m_0}\cdot...\cdot\binom{n_r}{m_r}\mod p.
	\]
	The special case $m=p^j$ gives the desired result.
\end{proof}

We can now check the assumptions of \Cref{thm1} for the $(e^{(p)}_j)_{j\geq 0}$ defined in \Cref{def:e_and_E}.
\begin{lem}\label{lem2} We have:
\begin{enumerate}
	\item For all $j\geq 0$: $e^{(p)}_j(\Zp^*)\subseteq \Zp$ and $e^{(p)}_0(\Zp^*)\subseteq \Zp^*$.
	\item For all $n\geq 1$ the reductions $(\bar{e}^{(p)}_{j})_{0\leq j\leq n-1}\in \Cont(\Zp^*,\Fp)$ are constant on all cosets of $1+p^n\Zp\subseteq\Zp^*$.
	\item For all $n\ge 1$, the map
	\[
	\begin{tikzcd}[row sep=tiny]
		\phi\colon \Zp^*/(1+p^n\Zp) \ar{r} & \Fp^*\times \Fp^{n-1} \\
		x \ar[mapsto]{r} & (\bar{e}^{(p)}_{0}(x),(\bar{e}^{(p)}_{j}(x))_{1\leq j\leq n-1})
	\end{tikzcd}
	\]
	is bijective.
\end{enumerate}
\end{lem}
\begin{proof}
		For $j\geq 1$ let us write for the set
		\[
			I_j:=\left\{-\left\lfloor\frac{p^j}{2}\right\rfloor, ..., p^j-1-\left\lfloor\frac{p^j}{2}\right\rfloor \right\}=\left\{ \tilde{a}\mid a\in(\mathbb{Z}/p^j\mathbb{Z})^\times\right\}
		\]
		of our chosen representatives
		and note the following equality in $\QQ[X]$:
		\begin{equation}\label{lemeq1}
			\begin{split}
				\binom{X+\lfloor \frac{p^j}{2} \rfloor}{p^j}& =\frac{1}{(p^j)!}\prod_{i\in I_j,(p,i)=1} (X-i)\prod_{i\in I_j,p|i} (X-i)=\\
				&=e^{(p)}_j(X)\prod_{i\in I_j,p|i} (X-i).
			\end{split}
		\end{equation}
		Also note
		\[
			\binom{X+\lfloor \frac{p^0}{2} \rfloor}{p^0}=X=e_0(X)
		\]
		for $j=0$. This allows us to deduce the lemma from the corresponding properties of the polynomials $\binom{X}{p^j}$:\par
		$(a)$ follows from \eqref{lemeq1} since $\binom{x}{p^j}\in\Zp$ and $\prod_{i\in I_j,p|i} (x-i) \in \Zp^*$ for all $x\in \Zp^*$.\par
		$(b)$ Fix $n\ge 1$, $x,y\in\Zp^*$ such that $xy^{-1}\in1+p^n\Zp$ and
		an index $0\leq j<n$. We need to check that $e^{(p)}_j(x)\equiv e^{(p)}_j(y)$ mod $p$. Our assumption implies that $x\equiv y$ mod $p^n$ , hence also $x+\lfloor \frac{p^j}{2}\rfloor \equiv y + \lfloor \frac{p^j}{2}\rfloor$ mod $p^n$, and \Cref{cor3} implies that 
\[ \binom{x+\lfloor\frac{p^j}{2}\rfloor}{p^j} \equiv 	\binom{y+\lfloor\frac{p^j}{2}\rfloor}{p^j} \mbox{ mod }p. \] 
Combining this with \eqref{lemeq1} and keeping in mind that $x\equiv y \mod p$ we get $e_j^{(p)}(x)\equiv 
e_j^{(p)}(y) \mod p$ and hence the claim.	\par

		$(c)$ To see the bijectivity of $\phi$, consider first the map
		\[
		\begin{tikzcd}[row sep=tiny]
			\psi_n \colon \Zp/p^n \Zp \ar{r} & \Fp^n\\
			x\ar[mapsto]{r} & \left(\binom{x+\lfloor \frac{p^j}{2} \rfloor}{p^j}\mod p\right)_{0\leq j\leq n-1}
		\end{tikzcd}
		\]
		which is well-defined by the above observation that $x\mapsto \binom{x}{p^j}$ for $0\leq j\leq n-1$ just depends on the coset $x+p^n\Zp$.
		We will show the bijectivity of $\psi_n$ by induction on $n$. The case $n=1$ is obvious. Let $n\geq 2$ and assume we have the bijectivity for all $i<n$. Assume we have $x,\tilde{x}\in \Zp/p^n\Zp$ with $\psi_n(x)=\psi_n(\tilde{x})$. The elements $x,\tilde{x}\in \Zp/p^n\Zp$ may be written uniquely in the form
		\begin{align*}
			x=\sum_{i=0}^{n-1} \beta_i p^i, \quad \beta_i\in\{ 0,...,p-1 \}\\
			\tilde{x}=\sum_{i=0}^{n-1} \tilde{\beta}_i p^i, \quad \tilde{\beta}_i\in\{ 0,...,p-1 \}.
		\end{align*}
		By induction we already know $\beta_i=\tilde{\beta}_i$ for $i<n-1$. Our aim is to show $\beta_{n-1}=\tilde{\beta}_{n-1}$, which will give the injectivity of $\psi_n$. Now we remark that \Cref{cor3} implies for integers $i,n\geq 0$
		\[
			\binom{y+ip^{n-1}}{p^{n-1}} \equiv i+ \binom{y}{p^{n-1}} \mod p,
		\]
		because the right hand side is the $(n-1)$-st $p$-adic digit of $y+ip^{n-1}$. Applied to our situation, this yields
		\[
			\binom{x+\lfloor \frac{p^{n-1}}{2} \rfloor}{p^{n-1}}\equiv \binom{\tilde{x}+\lfloor \frac{p^{n-1}}{2} \rfloor}{p^{n-1}} +\beta_{n-1}-\tilde{\beta}_{n-1} \mod p.
		\]
		So our assumption $\psi_n(x)=\psi_n(\tilde{x})$ implies $\beta_{n-1}\equiv \tilde{\beta}_{n-1} \mod p$ which shows $\beta_{n-1}=\tilde{\beta}_{n-1}$ as desired. So $\psi_n$ is injective and bijectivity follows by cardinality reasons.\par
		Now we can deduce the bijectivity of $\phi$. Consider the following diagram:
		\[
		\begin{tikzcd}
			\ZZ/p^n\ZZ=\Zp/p^n\Zp \ar{r}{\psi_n} & \Fp^n \\
			(\ZZ/p^n\ZZ)^*= \Zp^*/(1+p^n\Zp) \arrow[hook]{u} \ar{r}{\phi} & \Fp^*\times \Fp^{n-1} \arrow[hook]{u}{\iota}.
		\end{tikzcd}
		\]
		Here the map on the left is the natural inclusion and $\iota$ is defined by
		\[
		\begin{tikzcd}[row sep=tiny]
			\Fp^*\times \Fp^{n-1} \arrow[hook]{r}{\iota} & \Fp^n\\
			(x_0,x_1,...,x_{n-1}) \arrow[mapsto]{r} & (x_0,x_0x_1,x_0x_2,...,x_0x_{n-1}).
		\end{tikzcd}
		\]
		The commutativity of the diagram follows from \eqref{lemeq1} and implies the injectivity of $\phi$. The bijectivity follows since source and target of $\phi$ are finite sets of the same cardinality.
\end{proof}

This lemma immediately implies:
\begin{prop}\label{prop:onb}
	The family of polynomials $(E^{(p)}_{2n})_{n\geq 0}$ is an orthonormal basis for $\Cev(\Zp^*,\Qp)\cong \Cont(\G,\Qp)$, and $(E^{(p)}_{n})_{n\geq 0}$ is an orthonormal basis for $\Cont(\Zp^*,\Qp)$.
\end{prop}
\begin{proof}
	Follows from \Cref{thm1} and \Cref{lem2}.
\end{proof}

For studying sequences of moments starting with the $k$-th moment for some given
$k\ge 0$, we will need the following mild amplification.

\begin{cor}\label{cor1} We have
	\begin{enumerate}
		\item For fixed $k\geq 0$, the family of polynomials $(X^k E^{(p)}_{n}(X))_{n\geq 0}$ is an orthonormal basis of $\Cont(\Zp^*,\Qp)$.
		\item For every even $k\geq 0$, the family of polynomials $(X^k E^{(p)}_{2n}(X))_{n\geq 0}$ is an orthonormal basis of $\Cev(\Zp^*,\Qp)\cong\Cont(\G,\Qp)$.
	\end{enumerate}
	
\end{cor}
\begin{proof}
	$(a)$ By \Cref{lem:lem1} we are reduced to proving that the functions
	\[
		\left(x\mapsto x^k\bar{E}^{(p)}_{n}(x)\right)_{0\leq n\leq \varphi(p^i)-1}
	\]
	are linearly independent in the vector space $\mathrm{Maps}(\Zp^*/(1+p^i\Zp),\Fp)$. But this is clear from the linear independence of the functions
	\[
		\left(x\mapsto \bar{E}^{(p)}_{n}(x)\right)_{0\leq n\leq \varphi(p^i)-1}	.
	\]
The proof of $(b)$ is analogous. 
\end{proof}

We now obtain our first result characterizing sequences of moments in terms
of recursively defined and explicit congruences.

\begin{cor}\label{cor2} Fix some integer $m\geq 0$.	Let $(b_k)_{k\geq 2m}$ be an even sequence of $p$-adic integers, i.e. $b_k\in \Zp$ with $b_{k}=0$ for $k$ odd. Then the following are equivalent:
	\begin{enumerate}
	\item\label{it:exists-measure} There exists a measure $\mu\in \meas(\G,\Zp)$ with
			\[
				b_{2k}=\int_{\G}{x^{2k} d\mu}, \quad k\geq m.
			\]
	\item For each $k\geq m$ we have
			\[
				b_{2k}\equiv \sum_{i=m}^{k-1} c^{(k-m,p)}_{2(i-m)}b_{2i} \mod C_p(k-m)\Zp
			\]
			where $1/C_p(k-m)$ is the leading coefficient of $E^{(p)}_{2(k-m)}(X)$, i.e. $C_p(k-m):=\prod_{i\geq 0}(p^i!)^{a_i}$, where $2(k-m)=a_0 + \sum_{i\geq 1} a_i \varphi(p^i)$ is the expansion by $\varphi$-digits, and the $c^{(k-m,p)}_{2(i-m)}$ are integers determined by:
			\[
				  C_p(k-m) E^{(p)}_{2(k-m)}(X)X^{2m}=:X^{2k}-\sum_{i=m}^{k-1} c^{(k-m,p)}_{2(i-m)} X^{2i}.
			\]
	\end{enumerate}
\end{cor}

\begin{proof}
	Assume $(a)$. Then, for $k\geq m$ we have
	\[
			\frac{1}{C_p(k-m)}\left(b_{2k}-\sum_{i=m}^{k-1} c^{(k-m,p)}_{2(i-m)} b_{2i}\right)=\int_{\G}{E^{(p)}_{2(k-m)} (x)x^{2m}}\in \Zp,
	\]
	hence $(b)$ hold.\par
	Conversely, assume we are given a sequence $(b_k)_{k\geq 2m}$ as in $(b)$. We then define $p$-adic integers $\alpha_{2k}\in\Zp,k\geq m$, using $(b)$, by:
	\[
			\alpha_{2k}:=\frac{b_{2k}-\sum_{i=m}^{k-1} c^{(k-m,p)}_{2(i-m)}b_{2i}}{C_p(k-m)},\quad k\geq m.
	\]
	Set $\tilde{E}_k:=X^{2m}E^{(p)}_{k-2m}(X),k\geq 2m$, then $(\tilde{E}_{2k})_{k\geq m}$ is an orthonormal basis for $\Cont(\G,\Qp)$ by  \Cref{cor1}. So
	\[
			\tilde{E}_{2k} \mapsto \alpha_{2k},\quad k\geq m
	\]
	extends by linearity to a $\Zp$-linear map:
	\[
			\mu:\Cont(\G,\Zp)\rightarrow \Zp\mbox{, i.e. }\mu\in\meas(\G,\Zp).
	\]
	By construction, $\mu$ satisfies condition $(a)$.
\end{proof}

\begin{rem}\label{rem1} To check whether a given even sequence $(b_k)_{k\geq 2m}$ is a sequence of moments, there is at most one congruence condition to check for each $b_{2k},\, k\geq m$, and the congruence just depends on the $b_{2i}$ with $i<k$. In the case $p\neq 2$ the congruence condition for $b_{2k}$ is automatically true as long as  $2(k-m)<p-1$ since $v_p(C_p(k-m))=0$ in this range. In particular, a sequence $(b_k)_{k\ge 2m}$ of integers
satisfies \Cref{cor2}, \ref{it:exists-measure} {\em for all primes $p$ simultaneously} if and only if, for all $k\ge 2m$ the integer $b_{2k}$ satisfies a congruence
condition involving only the $b_{2i}$ ($m\leq i<k$), and {\em finitely many}
primes, see \Cref{thm2}, \ref{it:explicit-finite} for details, and see \Cref{ex:moments} for an explicit computation.
\end{rem}

For later reference, we give the condition \Cref{cor2}, \ref{it:exists-measure} a proper name.

\begin{defin}\label{def:measurecond}
For a prime $p$, some integer $m\geq 1$ and an even sequence $(b_k)_{k\geq 2m}$, i.e. $b_{2k+1}=0,k\geq m$, of $p$-adic integers $b_k\in\Zp$, define the following conditions:
	\begin{itemize}
		\item[\Bp] There exists a measure $\mu\in\meas(\Zp^*/\{\pm 1\},\Zp)$ such that $b_{2k}=\int_{\Zp^*/\{\pm 1\}}{x^{2k} d\mu(x)}$ for all $k\geq m$.
		\item[\Bpa] For every $c\in \Zp^*$ there exists a measure $\mu_c \in\meas(\Zp^*/\{\pm 1\},\Zp)$ such that $b_{2k}(1-c^{2k})=\int_{\Zp^*/\{\pm 1\}}{x^{2k} d\mu_c(x)}$ for all $k\geq m$.
	\end{itemize}
\end{defin}
The construction
\[
	\begin{tikzcd}[row sep=tiny]
		\meas(\G,\Zp) \ar{r} & \meas(\G,\Zp),\\
		\mu \ar[mapsto]{r} & \mu-c_{*}\mu
	\end{tikzcd}
\]
where $c_*$ is induced by multiplication with $c\in\G$, shows that condition \Bp\ implies condition \Bpa , because $c_*$ multiplies the $2k$-th moment by $c^{2k}$. 
We now check that conversely, $\Bpa$ does in fact imply $\Bp$.

\begin{prop}\label{prop_regularize} Fix a prime $p$ and let $c$ be a topological generator of $\G$.  Then,
the following sequence is exact	
\[
		\begin{tikzcd}[column sep=small]
			0 \ar{r} & \meas(\G,\Zp) \ar{r}{id-c_*} & \meas(\G,\Zp) \ar{r}{\mathrm{mass}} & \Zp \ar{r} & 0,
		\end{tikzcd}
		\]
		where $\mathrm{mass}(\mu):=\int_{\G}1\, d(\mu)$ is the total mass of the measure.
\end{prop}
\begin{proof}
First note that $\G$ is pro-cyclic, so there always exists a topological generator $c\in\G$.	We consider the following sequence:
		\begin{equation}\label{seq1}
			\begin{tikzcd}[row sep=tiny]
				0 \ar{r} & \Zp \ar{r}{\mathrm{incl}} & \Cont(\G,\Zp) \ar{r}{ \id-c^*} & \Cont(\G,\Zp)  \ar{r} & 0\, , \\
			\end{tikzcd}		
		\end{equation}
		where $c^*(f)(x):=f(cx)$, and incl is the inclusion of the constant functions.
		We claim that \eqref{seq1} is exact. The injectivity of $\mathrm{incl}$ and the inclusion $\Zp \subseteq \ker(\id-c^* )$ are clear. To see the exactness in the middle, consider some $f\in\Cont(\G,\Zp)$ with $ f-c^* f=0$. We get 
		\[
		f(x)=f(c^i x)\quad\forall x\in \G,i\in\ZZ,
		\]
		so $f$ is constant since it is constant on the dense subset $\{ c^i:i\in \ZZ\}\subseteq\G$.\par
		The first step in the proof of the surjectivity of $\id-c^*$ is to show that the cokernel of $\id-c^*$ is torsion-free: Therefore, assume given some $f\in \Cont(\G,\Zp)$ and $r\ge 0$ with  $p^rf\in \im(\id-c^*)$, i.e.
		\[
		\exists g\in \Cont(\G,\Zp) \text{ such that }  p^rf(x)=g(x)-g(cx)\quad \forall x\in \G.
		\]
		We have to show $f\in \im(\id-c^*)$. We get $g(x)\equiv g(cx) \mod p^r$ for all $x\in \G$ which implies
		\[
			 g(x)\equiv g(c^i x) \mod p^r \quad \forall x\in \G,i\in \ZZ,
		\]
		and hence that $g$ is constant $\mod p^r$, i.e. there exists $D\in \Zp$ with $g(x)\equiv D\mod p^r$.  Now we see that  $f=(\id-c^*)\left(\frac{g-D}{p^r}\right)\in \im(\id-c^*)$, as desired.\par
		Next we consider the following $\Zp$-submodules $A_n\subseteq C(\G,\Zp)$ for $n\ge 1$:
		\[
			A_n:=\{ P=\sum_{i=1}^n a_i X^{2i}\in\Qp[X]: P(\G)\subseteq \Zp  \}
		\]
		Observe that each $A_n$ is spanned by the first $n$ elements
		\[
			X^2E^{(p)}_{0}(X),...,X^2E^{(p)}_{2(n-1)}(X)
		\]
		of the orthonormal basis $(X^2E^{(p)}_{2k}(X))_{k\geq 0}$ of \Cref{cor1}. This implies that $A_n$ is a direct summand of $\Cont(\G,\Zp)$. Since $\coker(\id-c^*)$ is torsion-free this implies that also $A_n/(A_n\cap \im(\id-c^*))$ is torsion-free. But the $\Zp$-submodule of $A_n$ spanned by
		\[
				(1-c^2)X^2,...,(1-c^{2n}X^{2n})
		\]
		is of finite index in $A_n$ and contained in the image of $\id-c^*$ so it follows that $A_n= \im(\id-c^*)\cap A_n$ for all $n\geq 1$ since $A_n/(A_n\cap \im(\id-c^*))$ is torsion free. We get $\bigcup_{n\geq 1} A_n\subseteq \im(\id-c^*)$ and conclude $\im(\id-c^*)=\Cont(\G,\Zp)$ since $\bigcup_{n\geq 1} A_n$ contains an orthonormal basis as shown in \Cref{cor1}.\par
		Dualizing \eqref{seq1} gives the exactness of
		\[
		\begin{tikzcd}[row sep=tiny]
			0 \ar{r} & \meas(\G,\Zp) \ar{r}{id-c_*} & \meas(\G,\Zp) \ar{r}{\int_{\G}1 \,d(\cdot)} & \Zp,
		\end{tikzcd}
		\]
		and the surjectivity of $\mu \mapsto \int_{\G}1\, d(\mu)$ is obvious.		
\end{proof}

\begin{rem}
We sketch a different proof of \Cref{prop_regularize}. Note that $\G\cong \Delta\times \Zp$ for some cyclic group $\Delta$ of order $\frac{p-1}{2}$ if $p\neq2$ (resp. $1$ if $p=2$) and use the resulting isomorphism
	\[
			\meas(\G,\Zp)\cong \bigoplus_{\hat{\Delta}} \Zp\llbracket X \rrbracket
	\]
	given by Fourier-transformation of measures \cite{schneider-teitelbaum}. Then the exactness follows by a direct calculation of $\id-c_*$ on $\bigoplus \Zp\llbracket X \rrbracket$.
\end{rem}

\begin{cor}\label{cor_BpBpa}
	Conditions \Bp\ and \Bpa\ of \Cref{def:measurecond} are equivalent.
\end{cor}
\begin{proof}
	As noted above, \Bp\ implies \Bpa\ since $\mathrm{id}-c_*$ is multiplication by $(1-c^{2k})_{k\geq m}$ on moments.\par
	To see the converse implication, let $(b_k)_{k\geq 2m}$ be an even sequence of $p$-adic integers satisfying $\Bpa$. In particular, for a given $c\in\ZZ_p^\times$ which projects to a topological generator in $\G$ we get a measure $\mu_c$ satisfying
	\begin{equation}\label{eq:alpha}
			b_{2k}(1-c^{2k})=\int_{\G}{x^{2k}d\mu_c},\quad  k\geq m.
	\end{equation} 
We first claim that $\int_\G1\,d\mu_c=0$. We have $1=\lim_{r\to\infty} 
x^{\varphi(p^r)}$ in C$(\G,\ZZ_p)$, and hence 
\[ \int_\G1\,d\mu_c=\lim_{r\to\infty} \int_\G x^{\varphi(p^r)}\,d\mu_c\stackrel{\eqref{eq:alpha}}{=}\lim_{r\to\infty} (1-c^{\varphi(p^r)}) b_{\varphi(p^r)}=0, \]
because $c^{\varphi(p^r)}\to 1$. Now, \Cref{prop_regularize} implies the existence of a measure $\mu$ such that $\mu_c=\mu-c_*\mu$. It follows:
	\[
		(1-c^{2k})\int_{\G}{x^{2k}d\mu}=\int_{\G}{x^{2k}d\mu_c}\stackrel{\eqref{eq:alpha}}{=}(1-c^{2k})b_{2k} \, , \, \forall k\geq m.
	\]
Since $1-c^{2k}\neq 0$ for all $k\ge m$, we find $\int_\G x^{2k}\,d\mu=b_{2k}$
for all $k\ge m$, i.e. condition $\Bp$ holds.
\end{proof}

\subsection[Interpolation of integer sequences]{Simultaneous $p$-adic interpolation of integer sequences}\label{sec:simultaneous}

\subsubsection{Sequences with the Euler factor removed}

In this subsection, we study the following groups determined by simultaneous $p$-adic interpolation properties for all primes $p$.

\begin{defin}\label{def1}
	Fix some integer $m\geq 1$.	
	We will denote by $\Momeul$ the set of all integer sequences $(b_k)_{k\geq 2m}$ ($b_k\in\ZZ$), satisfying the following conditions:\footnote{The notation is to remind of {\em moments} of measures with the {\em Euler factor} removed.}
	\begin{enumerate}
		\item[\textbf{(A)}] We have $b_{2k+1}=0$ for $k\geq m$.
		\item[\textbf{(B)}] For all primes $p$, the sequence $((1-p^{k-1})b_k)_{k\geq 2m}$ satisfies  the equivalent conditions \Bpa\ and \Bp  (cf. \Cref{def:measurecond}).
	\end{enumerate}
	Obviously, $\Momeul$ is a group under addition of integer sequences.
\end{defin}

The aim of this section is the construction of an isomorphism of groups
	\[
		\begin{tikzcd}[row sep=tiny]
			\Phi_m \colon \ZZ^\NNo \ar{r}{\simeq}& \Momeul.
		\end{tikzcd}
	\]

\begin{rem}\label{rem:relevance}

In \cite[Proposition 5.19, iv)] {andohopkinsrezk}it is shown that the group of spectrum maps $[bstring,gl_1(\KO)]$ acts freely and transitively on the set of homotopy classes of $\mathbb{E}_\infty$-maps $\pi_0 \mathbb{E}_\infty (\MString,\KO)$. A trivialization of the torsor $\pi_0 \mathbb{E}_\infty (\MString,\KO)$ is given by the Atiyah-Bott-Shapiro orientation. The corresponding free and transitive action on the characteristic series of $\mathbb{E}_\infty$-String orientations is given by addition with sequences from $\Momeulzwei$. Put differently, we have a group isomorphism
\[ [ bstring, gl_1(\KO) ] \simeq \Momeulzwei\, (\stackrel{\Phi_4}{\simeq}\mathbb{Z}^{\mathbb{N}_0}). \]
The analogous result with $String$ replaced by $Spin$ and $\Momeulzwei$ replaced with $\Momeuleins$ also holds true.
\end{rem}

Our previous work leads to the following description of $\Momeul\subseteq
\mathbb{Z}^{\mathbb{N}_{\ge 2m}}$ by recursive congruences.

\begin{thm}\label{thm2}
For an even sequence $(b_k)_{k\geq 2m}$ of integers $b_k\in\mathbb{Z}$, the following are equivalent:
\begin{enumerate}
	\item We have $(b_k)_{k\geq 2m} \in \Momeul$.
	\item\label{it:explicit-finite} For all $k\geq m$ and all primes $p$ such that $p-1\leq 2(k-m)$ we have
			\[
				(1-p^{2k-1})b_{2k}\equiv \sum_{i=m}^{k-1} c^{(k-m,p)}_{2(i-m)}(1-p^{2i-1})b_{2i} \mod p^{v_p(C_p(k-m))},
			\]
			where $C_p(k-m)$ and	the $c^{(k-m,p)}_i$ are defined as in \Cref{cor2}.
	\item For every prime $p$ there is a measure $\mu_p\in\meas(\G,\Zp)$ such that for all
	$k\ge 2m$ \[ \int_{\G} x^{2k}\,d\mu_p(x)=(1-p^{2k-1})b_{2k}.\]	
		
\end{enumerate}
\end{thm}
\begin{proof}

The equivalence of conditions $(a)$ and $(b)$ is clear from \Cref{cor2} (for $(1-p^{k-1}b_k)$ in place of $(b_k)$ there) and \Cref{rem1}. Condition $(c)$ is the definition of $(a)$ (see \Cref{def1} and \Cref{def:measurecond}), restated here for the ease of reference.
\end{proof}

Next, we we will construct for every $m\geq 1$ an isomorphism of groups 
\[
		\begin{tikzcd}[row sep=tiny]
			\Phi_m\colon\ZZ^\NNo \ar{r} & \Momeul.
		\end{tikzcd}
\]
It will be given by multiplication with an integer matrix $(\Phi^{(m)}_{i,j})_{i,j\geq 0}\in \ZZ^{\NNo\times \NNo}$ with
\begin{equation}\label{eq:many-zeroes}
	\begin{split}
		\Phi^{(m)}_{2i+1,j}&=0\quad \forall i,j\geq 0\\
		\Phi^{(m)}_{2i,j}&=0\quad \forall i\geq0,j>i,
	\end{split}
\end{equation}
i.e.
\[
	(\Phi^{(m)}_{i,j})_{i,j\geq 0}=
	\left(\begin{matrix}
		\Phi^{(m)}_{0,0} & 0 & 0 & 0 & \cdots\\
		0 & 0 & 0 & 0 & \cdots \\
		\Phi^{(m)}_{2,0} & \Phi^{(m)}_{2,1} & 0 & 0 & \cdots\\
		0 & 0 & 0 & 0 & \cdots \\
		\Phi^{(m)}_{4,0} & \Phi^{(m)}_{4,1} & \Phi^{(m)}_{4,2} & 0 & \cdots\\
		0 & 0 & 0 & 0 & \cdots \\
		\vdots & \vdots & \vdots & \vdots & \ddots
	\end{matrix}\right)
\]
via
\[
	\Phi_m\left((l_i)_{i\geq 0}\right):=\left( \sum_{j\geq 0} \Phi^{(m)}_{k-2m,j} l_j \right)_{k\geq 2m}=(\Phi^{(m)}_{0,0}\cdot l_0,0,\Phi^{(m)}_{2,0}\cdot l_0+\Phi^{(m)}_{2,1}\cdot l_1,0,\ldots).
\]
The rows $\Phi^{(m)}_{i,*}$ will be defined inductively over $i$. For $i=0$ and $i=1$ define
\begin{align*}
		\Phi^{(m)}_{0,*}:&=(1,0,0,0,\cdots)\\
		\Phi^{(m)}_{1,*}:&=(0,0,0,0,\cdots).
\end{align*}
Let $k>0$ and assume we have already defined $\Phi^{(m)}_{i,*}$ for all $i<2k$. Let $S_k$ be the finite set of primes $S_k:=\{p \text{ prime}\mid p-1\leq 2k\}$. For $0\leq j < k$ let $\Phi^{(m)}_{2k,j}$ be the unique integer with $0\leq \Phi^{(m)}_{2k,j}<\prod_{p\in S_k} p^{v_p(C_p(k))}$ satisfying
\begin{equation}\label{eq:defPhi1}
	(1-p^{2(k+m)-1})\Phi^{(m)}_{2k,j} \equiv \sum_{i=0}^{k-1} c^{(k,p)}_{2i} (1-p^{2(i+m)-1})\Phi^{(m)}_{2i,j} \mod p^{v_p(C_p(k))}
\end{equation}
for all $p\in S_k$. 
\footnote{Uniqueness and existence of $\Phi_{2k,j}^{(m)}$ follow from the Chinese remainder 
theorem, using that $1-p^{2(k+m)-1}$ is a unit mod $p^{v_p(C_p(k))}$.}
Further set 
\begin{equation}\label{eq:defPhi2}
	\begin{split}
		\Phi^{(m)}_{2k,k}:&= \prod_{p\in S_k} p^{v_p(C_p(k))}\\
		\Phi^{(m)}_{2k,j}:&= 0 \quad j>k
	\end{split}
\end{equation}
and
\[
	\Phi^{(m)}_{2k+1,*}:=(0,0,0,0,\cdots).
\]
With these constructions in place, we have the following:

\begin{thm}\label{thm3}
	Let $m\geq1$ be an integer. The above map 
	\[
		\begin{tikzcd}[row sep=tiny]
			\Phi_m \colon \ZZ^\NNo \ar{r} & \Momeul \subseteq \ZZ^{\NN_{\geq 2m}}\\
			(l_i)_{i\geq 0} \ar[mapsto]{r} & \left( \sum_{j\geq 0} \Phi^{(m)}_{k-2m,j} l_j \right)_{k\geq 2m}
		\end{tikzcd}
	\]
	is a well-defined isomorphism of groups.
\end{thm}
\begin{proof}
	First we will show the well-definedness: Let $(l_i)_{i\geq 0}\in\ZZ^\NNo$ be given and set
	\[
		(b_k)_{k\geq 2m}:=\Phi_m((l_i)_{i\geq 0}).
	\]
	To see that $(b_k)_{k\ge 2m}\in\Momeul$, by \Cref{thm2} it suffices to show that for all $k\ge m$ and all primes $p$ the equation
	\[
		(1-p^{2k-1})b_{2k}\equiv \sum_{i=m}^{k-1} c^{(k-m,p)}_{2(i-m)}(1-p^{2i-1})b_{2i} \mod p^{v_p(C_p(k-m))}
	\]
	holds. By definition of $b_{2k}$ and \eqref{eq:defPhi2} we have for $k\ge m$ and $p\in S_{k-m}$
	\[
	\begin{split}
		(1-p^{2k-1})b_{2k}&=(1-p^{2k-1})\sum_{j=0}^{k-m} \Phi^{(m)}_{2(k-m),j} l_j \equiv\\
		& \equiv (1-p^{2k-1})\sum_{j=0}^{k-m-1} \Phi^{(m)}_{2(k-m),j} l_j \mod p^{v_p(C_p(k-m))}.
	\end{split}
	\]
	We further use \eqref{eq:defPhi1} to get the desired congruence for all $k\ge m$ and $p\in S_{k-m}$:
	\[
	\begin{split}
		(1-p^{2k-1})b_{2k} & \equiv (1-p^{2k-1})\sum_{j=0}^{k-m-1} \Phi^{(m)}_{2(k-m),j} l_j \equiv\\
		& \stackrel{(\ref{eq:defPhi1})}{\equiv} \sum_{j=0}^{k-m-1}\left( \sum_{i=m}^{k-1} c_{2(i-m)}^{(k-m,p)}(1-p^{2i-1})\Phi^{(m)}_{2(i-m),j} \right) l_j \equiv \\
		& \equiv \sum_{i=m}^{k-1} c_{2(i-m)}^{(k-m,p)}(1-p^{2i-1})b_{2i}  \mod p^{v_p(C_p(k-m))}.
	\end{split}
	\]
	Note that this congruence is vacuously satisfied for all primes $p\not\in S_{k-m}$ by \Cref{rem1}.\par
	It is clear that $\Phi_m$ is a homomorphism. The injectivity follows from the fact that $(\Phi^{(m)}_{2i,j})_{i,j\geq 0}$ is upper triangular with non-vanishing values $\Phi^{(m)}_{2k,k}=\prod_{p\in S_k} p^{v_p(C_p(k))}$ on the diagonal.\footnote{Strictly speaking, rather a line of slope two than the diagonal.}\par
	To see the surjectivity of $\Phi_m$, we will construct an inverse. Let $(b_k)_{k\geq 2m}\in\Momeul$ be given. Set $l_0:=b_{2m}$. Assume $n>0$ and that we already have defined $l_i$ for $i<n$ in such a way that
	\begin{equation}\label{eq:thm3_induction}
		b_{2(i+m)}=\sum_{j= 0}^{i} \Phi^{(m)}_{2i,j} l_j
	\end{equation}
	holds for all $ i<n$. For the following chain of congruences we use \Cref{thm2} applied to $(b_k)_{k\geq2m}$, the induction hypothesis (\ref{eq:thm3_induction}) and the defining equations of $\Phi^{(m)}_{i,j}$:
	\[
	\begin{split}
		(1-p^{2(n+m)-1})b_{2(n+m)} & \stackrel{\text{Thm.}\ref{thm2},\ref{it:explicit-finite}}{\equiv} \sum_{i=m}^{n+m-1} c_{2(i-m)}^{(n,p)}(1-p^{2i-1})b_{2i} \equiv \\
		& \stackrel{i-m\mapsto i}{\equiv} \sum_{i=0}^{n-1} c_{2i}^{(n,p)}(1-p^{2(i+m)-1})b_{2(i+m)} \equiv \\
		& \stackrel{(\ref{eq:thm3_induction})}{\equiv} \sum_{i=0}^{n-1} c_{2i}^{(n,p)}(1-p^{2(i+m)-1}) \left(\sum_{j=0}^{i}\Phi^{(m)}_{2i,j} l_j\right)  \equiv\\
		& \stackrel{\eqref{eq:defPhi2}}{\equiv} \sum_{j=0}^{n-1} \left(\sum_{i=0}^{n-1} c_{2i}^{(n,p)}(1-p^{2(i+m)-1}) \Phi^{(m)}_{2i,j}\right) l_j \equiv\\
		& \stackrel{\eqref{eq:defPhi1}}{\equiv} (1-p^{2(n+m)-1})\sum_{j=0}^{n-1} \Phi^{(m)}_{2n,j} l_j \mod p^{v_p(C_p(n))}.
	\end{split}
	\]
After canceling the $p$-adic unit $(1-p^{2(n+m)-1})$,
	this congruence shows that
	\[
		l_n:=\frac{b_{2(m+n)}-\sum_{j=0}^{n-1} \Phi^{(m)}_{2n,j} l_j }{\prod_{p\in S_n} p^{v_p(C_p(n))} }\in\ZZ
	\]
	is an integer. The desired equation
	\[
		b_{2(n+m)}=\sum_{j= 0}^{n} \Phi^{(m)}_{2n,j} l_j
	\]
	follows immediately from the definition of $l_n$. This construction gives a sequence $(l_i)_{i\geq0}$ which satisfies 
	$\Phi_m((l_i)_{i\geq0})=(b_k)_{k\geq 2m}$ by \eqref{eq:thm3_induction}.
\end{proof}

\begin{exmpl}\label{ex:moments}
	In this example we will calculate the first few rows of  the isomorphism
	\[ \Phi_1\colon\mathbb{Z}^{\mathbb{N}_0}\xrightarrow{\simeq}\mathrm{Mom}_{\ge 2}^{\mathrm{Euler}}\left( \simeq [bspin,gl_1(\KO)]\right), \]
	i.e. of the matrix
\[	
	(\Phi^{(1)}_{i,j})_{i,j\geq 0}=
	\left(\begin{matrix}
		\Phi^{(1)}_{0,0} & 0 & 0 & 0 & \cdots\\
		0 & 0 & 0 & 0 & \cdots \\
		\Phi^{(1)}_{2,0} & \Phi^{(1)}_{2,1} & 0 & 0 & \cdots\\
		0 & 0 & 0 & 0 & \cdots \\
		\Phi^{(1)}_{4,0} & \Phi^{(1)}_{4,1} & \Phi^{(1)}_{4,2} & 0 & \cdots\\
		0 & 0 & 0 & 0 & \cdots \\
		\vdots & \vdots & \vdots & \vdots & \ddots
	\end{matrix}\right).
\]
The calculation depends on the knowledge of sufficiently many of the polynomials $E_{2k}^{(p)}$ (cf. \Cref{def:e_and_E}). It turns out it suffices to know $E_{2k}^{(p)}$ for $p=2,3,5$ and $2k\leq 4$. Recall that $E_{2k}^{(p)}$ was defined by extension of $\varphi$-digits from the $e^{(p)}_{j\geq 0}$ as defined in \Cref{def:e_and_E}. For the $e^{(p)}_{j}$ we have:
	\begin{align*}
			e^{(2)}_0&=X,\quad e^{(2)}_1=\frac{X+1}{2},\quad e^{(2)}_2=\frac{X^2-1}{4!},\quad  e^{(2)}_3=\frac{(X^2-1)(X^2-9)}{8!},\\
			e^{(3)}_0&=X,\quad e^{(3)}_1=\frac{X^2-1}{3!},\\
			e^{(5)}_0&=X,\quad e^{(5)}_1=\frac{(X^2-1)(X^2-4)}{5!}.
	\end{align*}
	Now we can calculate $E_{2k}^{(p)}$ for $k=0,1,2$ and $p=2,3,5$:
	\begin{align*}
			E^{(2)}_0&=1,& E^{(2)}_2=&e^{(2)}_2,& E^{(2)}_4=&e^{(2)}_3, \\
			E^{(3)}_0&=1,& E^{(3)}_2=&e^{(3)}_1,& E^{(3)}_4=&(e^{(3)}_1)^2, \\
			E^{(5)}_0&=1,& E^{(5)}_2=&(e^{(5)}_0)^2,& E^{(5)}_4=&e^{(5)}_1.
	\end{align*}
	Expanding the above expressions for $E_{2k}^{(p)}$ will give the coefficients $c^{(k,p)}_{2i}$ and the inverse of the leading coefficient $C_p(k)$, which are needed for the following computations:\par
	The first two rows are given by
	\[
	\begin{split}
			\Phi^{(1)}_{0,*}:=(1,0,0,0,\cdots)\\
			\Phi^{(1)}_{1,*}:=(0,0,0,0,\cdots)
	\end{split}
	\]
	Next, let us compute $\Phi^{(1)}_{2k,*}$ for $k=1$. In this case the relevant set of primes $S_k:=\{ p\,\,\mathrm{ prime }\,\mid\, p-1\leq 2k\}$ is given by $S_1=\{2,3\}$. The polynomials
	\[
		X^2 E^{(2)}_2=\frac{X^4-X^2}{4!}=\frac{X^4-c^{(1,2)}_0 X^2}{C_2(1)}
	\]
	and
	\[
		X^2 E^{(3)}_2=\frac{X^4-X^2}{3!}=\frac{X^4-c^{(1,3)}_0 X^2}{C_3(1)}
	\]
	give:
	\begin{align*}
		C_2(1)=&4!,&C_3(1)=&3! \\
		c^{(1,2)}_0=&1,&c^{(1,3)}_0=&1.
	\end{align*}
	Hence the relevant modulus is (cf. \eqref{eq:defPhi1})
		\[ \prod_{p\in S_1}p^{v_p(C_p(1))}=8\cdot 3=24.\]
	Now $\Phi^{(1)}_{2,0}$ is defined as the unique integer with $0\leq \Phi^{(1)}_{2,0}< 2^3\cdot 3$ satisfying
	\[
		(1-p^{3})\Phi^{(1)}_{2,0} \equiv c^{(1,p)}_{0} (1-p)\Phi^{(1)}_{0,0} \mod p^{v_p(C_p(1))}
	\]
	for $p\in\{2,3\}$. Recalling that $\Phi_{0,0}^{(1)}=1$, this integer is easily computed to be $\Phi^{(1)}_{2,0}=7$. We obtain
\[
	\begin{split}
			\Phi^{(1)}_{2,*}=(7,24,0,0,\cdots)\\
			\Phi^{(1)}_{3,*}=(0,0,0,0,\cdots)
	\end{split}
	\]
	for the second and the third row.\par
	For $k=2$, we have
	\begin{equation*}
		S_2:=\{ p \text{ prime}\mid p-1\leq 4 \}=\{2,3,5\}
	\end{equation*}
	The coefficients of
	\begin{equation*}
	\begin{split}
		X^2 E^{(2)}_4=&\frac{X^6 - 10X^4 + 9X^2}{8!}=\frac{X^6 - c^{(2,2)}_2 X^4 -c^{(2,2)}_0 X^2}{C_2(2)}\\
		X^2 E^{(3)}_4=&\frac{X^6 - 2X^4 + X^2}{(3!)^2}=\frac{X^6 - c^{(2,3)}_2 X^4 -c^{(2,3)}_0 X^2}{C_3(2)}\\
		X^2 E^{(5)}_4=&\frac{X^6 - 5X^4 + 4X^2}{5!}=\frac{X^6 - c^{(2,5)}_2 X^4 -c^{(2,5)}_0 X^2}{C_5(2)}\\
	\end{split}
	\end{equation*}
	determine
	\begin{equation*}
			C_2(2)=8!,\quad C_3(2)=(3!)^2, \quad C_5(2)=5!
	\end{equation*}
	and:
	\begin{align*}
			c^{(2,2)}_0=&-9,&c^{(2,2)}_2=&10 \\
			c^{(2,3)}_0=&-1,&c^{(2,3)}_2=&2 \\
			c^{(2,5)}_0=&-4,&c^{(2,5)}_2=&5.
	\end{align*}	
	
	Hence the relevant modulus is
	
	\[ \prod_{p\in S_2} p^{v_p(C_p(2))}=2^7\cdot 3^2\cdot 5 = 5760.\]
	
	Solving the linear congruences for $\Phi^{(1)}_{4,0}$ and $\Phi^{(1)}_{4,1}$ gives
	\[
		\Phi^{(1)}_{4,0}=511 \mbox{ and } \Phi^{(1)}_{4,1}=4080.
	\]
	The fourth and fifth row are thus given by
	\[
	\begin{split}
			\Phi^{(1)}_{4,*}&=(511,4080,5760,0,\cdots)\\
			\Phi^{(1)}_{5,*}&=(0,0,0,0,\cdots).
	\end{split}
	\]
	All together we obtain
	\[
		(\Phi^{(1)}_{i,j})_{i,j\geq 0}=
		\left(\begin{matrix}
			1 & 0 & 0 & 0 & \cdots\\
			0 & 0 & 0 & 0 & \cdots \\
			7 & 24 & 0 & 0 & \cdots\\
			0 & 0 & 0 & 0 & \cdots \\
			511 & 480 & 5760 & 0 & \cdots\\
			0 & 0 & 0 & 0 & \cdots \\
			\vdots & \vdots & \vdots & \vdots & \ddots
		\end{matrix}\right)
	\]

	Combining this with \Cref{rem:relevance} yields the existence of a unique group isomorphism
	\[ \phi: \mathbb{Z}^{\mathbb{N}_0}\stackrel{\simeq}{\longrightarrow} [bspin, gl_1(KO)] \]
	such that for a given $(l_0,\ldots)\in\mathbb{Z}^{\mathbb{N}_0}$, defining
	$f:=\phi((l_0,\ldots))$, $(b_2,b_4,\ldots):=\Phi_1((l_0,\ldots))$ and (abusively) 
	$\beta$ the Bott element in both $\pi_*(bspin)$ and $\pi_*(gl_1(\KO))\otimes\mathbb{Q}\simeq \pi_*(\KO)\otimes\mathbb{Q}$, we have
	
	 \[ f(\beta^k)= b_k\cdot\beta^k\, , \, k\ge 2.\]
	 
	  It seems remarkable that {\em every} integer $l_0=b_2$ can be realized as the
	  multiplication on the bottom cell of some map $f: bspin\to gl_1(KO)$. 
	  This result also vastly generalizes the congruence $m\equiv 7n\,(12)$ from \cite[Rem. 6.2.3]{adams_loop}, which Adams used to show that there is
	  no {\em integral} equivalence $BO^\oplus\simeq BO^\otimes$ of $H$-spaces.
\end{exmpl}

\subsubsection[Sequences of total mass zero]{Sequences of total mass zero}

In the course of describing all $\mathbb{E}_\infty$-String orientations of $\tmf$, we will be concerned with the following group of integer sequences:
\begin{defin}\label{def0}
	Fix some integer $m\geq 1$.	The set of all integer sequences $(b_k)_{k\geq 2m}$ satisfying the following conditions will be denoted by $\Momzero$:\footnote{The notation means to suggest {\em moments} starting in weight $2m$ of measures with {\em total mass zero}.}
	\begin{enumerate}
		\item[\textbf{(A)}] We have $b_{2k+1}=0$ for $k\geq m$.
		\item[\textbf{(C)}] For each prime $p$, the sequence $(b_k)_{k\geq 2m}$ satisfies the equivalent conditions \Bp\ and \Bpa\ and in addition, the total mass of the corresponding measure is zero, to wit: for all primes $p$, there exists a measure $\mu_p\in\meas(\G,\Zp)$ such that
		\begin{align*}
			b_{2k}&=\int_{\G}{x^{2k}d\mu_p(x)}, k\geq m \text{ and }\\
					0 &=\int_{\G}{1 \, d\mu_p(x)}.
		\end{align*}
	\end{enumerate}
	Obviously $\Momzero$ is a group under addition of integer sequences.
\end{defin}

\begin{rem}
	Note that, unlike in \Cref{def1}, there is no "Euler-factor" $(1-p^{k-1})_{k\geq 2m}$ "removed" in condition \textbf{(C)}. The additional condition that the total mass of all appearing measures has to be zero has the following meaning for the sequence: Given $(b_k)_{k\geq 2m}\in \Momzero$ we have $b_{\varphi(p^r)}\rightarrow 0$ in $\Zp$ as $r\rightarrow \infty$ for every prime $p$.  
\end{rem}
Analogously to \Cref{thm2} in the case $\Momeul$, we get the following description of
$\Momzero\subseteq\mathbb{Z}^{\mathbb{N}_{\ge 2m}}$ by recursive congruences.
Recall that $\widehat{\ZZ}:=\lim_n \ZZ/n\ZZ$ denotes the pro-finite completion of the integers.

\begin{thm}\label{thm4}
For fixed $m\ge 1$ and an even sequence $(b_k)_{k\geq 2m}$ of integers, the following are equivalent:
\begin{enumerate}
	\item We have $(b_k)_{k\geq 2m} \in \Momzero$.
	\item For $0\leq i\leq 2m-1$ there exist (uniquely determined) $b_i\in \widehat{\ZZ}$ such that
	\[
	b_0=0, b_{2i+1}=0\text{ for }0\leq i\leq m-1
	\]
	and	such that for all $k\geq 1$ and all primes $p$ with $p-1\leq 2k$ we have the following congruence in $\widehat{\ZZ}$:
			\[
				b_{2k}\equiv \sum_{i=0}^{k-1} c^{(k,p)}_{2i}b_{2i} \mod p^{v_p(C_p(k))}\widehat{\ZZ}
			\]
	where $C_p(k)$ and	the $c^{(k,p)}_i$ are defined as in \Cref{cor2}.
\end{enumerate}
\end{thm}

\begin{proof}
	Assume $(b_k)_{k\geq 2m} \in \Momzero$ is given. By \Cref{def0}, for each prime $p$ there exists a measure $\mu_p\in\meas(\G,\ZZ_p)$  with
	\begin{align*}
			b_{2k}&=\int_{\G}{x^{2k}d\mu_p(x)}\text{ for all }k\geq m \text{ and }\\
			0&=\int_{\G}{1 \, d\mu_p(x)}.
	\end{align*}
	For $0\leq i\leq 2m-1$, we define elements $b_i\in\widehat{\ZZ} $ by specifying their $\Zp$-components 
	\[
		b_i=(b_{i,p})_p\in   \prod_{p \text{ prime}}\Zp \cong\widehat{\ZZ}
	\]
	for all primes $p$.	Namely, for an integer $0\leq i\leq m-1$ and a prime $p$ define:
	\begin{align*}
		b_{2i,p}:&= \int_{\G}{ x^{2i} d\mu_p(x) },\\
		b_{2i+1,p}:&=0.
	\end{align*}
	Note that $b_0=0$ by definition of $\Momzero$. The congruences at $p$ in condition $(b)$ between the $b_k$ are exactly the congruences satisfied by moments of measures on $\G$, see \Cref{cor2} (for $m=0$). The uniqueness  of the $b_i$ ($0\leq i \leq 2m-1$) follows since any tail of moments of a measure on $\G$ determines the measure uniquely.\par
	Conversely, assume the existence of $b_i\in \widehat{\ZZ}$ for $0\leq i\leq 2m-1$ as in $(b)$. For each prime $p$ and $i\ge 0$ define $b_{i,p}$ as the image of $b_i$ under the projection
	\[
		\begin{tikzcd}
			\widehat{\ZZ}\ar[two heads]{r} & \Zp.
		\end{tikzcd}
	\]
	Then, for each prime $p$ the congruences in $(b)$ ensure the assumptions of \Cref{cor2}, $(b)$ so \Cref{cor2} gives a measure $\mu_p\in\meas(\G,\ZZ_p)$ with
	\[
		b_{2k,p}=\int_{\G}x^{2k}d\mu_p(x)\,\, (k\ge 0) 
	\]
	and
	\[
		0=b_{0,p}=\int_{\G}1\,d\mu_p(x) , 
	\]
hence we have $(b_k)_{k\ge 2m}\in\Momzero$. (Note that for
$i\ge 2m$, the element $b_{i,p}\in\ZZ\subseteq\widehat{\ZZ}$
is independent of $p$).
\end{proof}

This result allows us to construct a bijection\footnote{The maps $\Psi_m^{(0)}$ will not be additive. Fortunately, this will not be required in any of our applications.} for $m\geq 1$
\[
		\begin{tikzcd}[row sep=tiny]
			\Psi_m^{(0)} \colon \widehat{\ZZ}^{m-1} \times \ZZ^{\NN_{\geq m}} \ar{r}{} & \Momzero
		\end{tikzcd}
\]
as follows: Let $(l_n)_{n\geq 1}\in \widehat{\ZZ}^{m-1} \times \ZZ^{\NN_{\geq m}}$, i.e. $l_i\in \widehat{\ZZ}$ for $1\leq i\leq m-1$ and $l_i\in\ZZ$ for $i\geq m$, be given. For each $k\geq 1$ define a set of primes $S_k$ by:
\[	
	S_k:=\{ p \text{ prime}\mid p-1\leq 2k \}
\]
and set:
\begin{align*}
	b_0:=0 \\
	b_1:=0
\end{align*}
Now define $b_i,i\geq 2$ recursively. Let $k\geq 1$ and assume we have already defined $b_i$ for all $i<2k$ let $0\leq b< \prod_{p\in S_k} p^{v_p(C_p(k))}$ be the unique integer satisfying
\[
	b \equiv \sum_{i=0}^{k-1} c^{(k,p)}_{2i}b_{2i} \mod p^{v_p(C_p(k))}\widehat{\ZZ}	\quad \forall p\in S_k,
\]
where existence and uniqueness follow from the Chinese remainder theorem. Define 
\[
	b_{2k}:=b+l_k  \prod_{p\in S_k} p^{v_p(C_p(k))}\in
	\begin{cases}
		\widehat{\ZZ} & k<m\\
		\ZZ & k\geq m
	\end{cases}
\]
and set $b_{2k+1}:=0$. This gives a sequence $(b_k)_{k\geq 0}\in \widehat{\ZZ}^{2m}\times\ZZ^{\NN_{\geq 2m}}$. Now ignoring the terms  of index $k<2m$ gives an integer sequence $(b_k)_{k\geq2m}$ which is contained in $\Momzero$ by \Cref{thm4} and by construction. Setting $\Psi^{(0)}_m((l_k)_{k\geq 1}):=(b_k)_{k\geq 2m}$ defines the map 
\[
		\begin{tikzcd}[row sep=tiny]
		\Psi^{(0)}_m \colon 	\widehat{\ZZ}^{m-1} \times \ZZ^{\NN_{\geq m}} \ar{r} & \Momzero.
		\end{tikzcd}
\]
We can now prove:
\begin{thm}\label{thm5}
	For a given integer $m\geq 1$, the above map  
		\[
			\begin{tikzcd}[row sep=tiny]
			\Psi_m^{(0)}	\colon \widehat{\ZZ}^{m-1} \times \ZZ^{\NN_{\geq m}} \ar[r] & \Momzero
			\end{tikzcd}
	\]
	is a bijection.
\end{thm}

\begin{proof}
	It is clear from the construction of $\Psi^{(0)}_m$ and \Cref{thm4} that $\Psi^{(0)}_m$ is well defined. Bijectivity follows from the following construction of an inverse map: Given $(b_k)_{k\geq 2m}\in \Momzero$ let $b_i\in\widehat{\ZZ}$ for $ 0\leq i\leq 2m-1$ be the uniquely determined elements as in \Cref{thm4}. For $k\geq 1$ let $0\leq b< \prod_{p\in S_k} p^{v_p(C_p(k))}$ be the unique integer with
	\[
		b\equiv b_{2k} \mod \prod_{p\in S_k} p^{v_p(C_p(k))}\widehat{\ZZ}.
	\]
	Setting 
	\[
		l_k:=\frac{b_{2k}-b}{\prod_{p\in S_k} p^{v_p(C_p(k))}} \in
		\begin{cases}
			\widehat{\ZZ} & k<m\\
			\ZZ & k\geq m
		\end{cases}
	\]
	provides the inverse to $\Psi_m^{(0)}$
	\[
			\begin{tikzcd}[row sep=tiny, column sep=large]
				\widehat{\ZZ}^{m-1} \times \ZZ^{\NN_{\geq m}}  & \Momzero \ar[l,swap," \left( \Psi_m^{(0)} \right)^{-1}"]\\
				(l_k)_{k\geq 0}  & \ar[l,mapsto] (b_k)_{k\geq 2m}.  \\
			\end{tikzcd}
	\]
\end{proof}

\section{Topological applications}\label{sec:top_app}

\subsection[Orientations of KO]{$\mathbb{E}_\infty$-String and $\mathbb{E}_\infty$-Spin orientations of KO}\label{sec:string_of_ko}

In \cite{andohopkinsrezk} M.~Ando, M.J.~Hopkins and C.~Rezk determine the set of all homotopy classes of $\mathbb{E}_\infty$-ring maps from $\MString$ (resp. $\MSpin$) to the real $K$-Theory spectrum $\KO$ in terms of even sequences of rational numbers satisfying a specific interpolation property at every prime $p$. After recalling their result we will give an explicit bijection between $\ZZ^\NNo$ and the set of these $\mathbb{E}_\infty$-orientations of $\KO$. The set $\mathrm{C}_{\mathrm{string}}(gl_1(\KO))$ is defined as the image of the characteristic map $b$, see \cite[Def. 5.11.]{andohopkinsrezk},
\[
		\begin{tikzcd}[row sep=tiny]
			b\colon \pi_0 \mathbb{E}_\infty (\MString,\KO) \ar{r} & \prod_{k\geq 4} \pi_{2k}(gl_1(\KO)) \otimes \QQ,
		\end{tikzcd}	
\]
and we define $\mathrm{C}_{\mathrm{spin}}(gl_1(\KO))$ in the same way, replacing $\MString$ by $\MSpin$ and replacing the condition
``$k\ge 4$'' by the condition ``$k\ge 2$''.
Recall that the Bernoulli numbers $B_k\in\QQ$ are defined by a generating series
$\frac{t}{e^t-1}=:\sum\limits_{k\ge 0}B_k\cdot\frac{t^k}{k !}$, giving for example 
$B_2=\frac{1}{6}$, $B_4=\frac{-1}{30}$ and $B_{2k+1}=0$ for all $k\ge 1$, cf. 

The main result of \cite{andohopkinsrezk} about $\mathbb{E}_\infty$-orientations of $\KO$ is as follows:
\begin{thm}[{\cite[Theorem 6.1]{andohopkinsrezk}}]\label{thm_ahs}
	The characteristic maps induce bijections
	\[
		\pi_0 \mathbb{E}_\infty (\MString,\KO)\xrightarrow{\simeq}\mathrm{C}_{\mathrm{string}}(gl_1(\KO))
	\]
	resp.
	\[
		\pi_0 \mathbb{E}_\infty (\MSpin,\KO)\xrightarrow{\simeq} \mathrm{C}_{\mathrm{spin}}(gl_1(\KO))
	\]
	and identify $\mathrm{C}_{\mathrm{string}}(gl_1(\KO))$ (resp. $\mathrm{C}_{\mathrm{spin}}(gl_1(\KO))$)
	with the set of sequences $(b_k)_{k\geq 4}$ (resp. $(b_k)_{k\geq 2}$), $b_k\in\QQ$ satisfying the following conditions:
	\begin{enumerate}
	\item We have $b_{2k+1}=0$ for $k\geq 2$ (resp. $k\geq 1$).
	\item We have $b_k\equiv -\frac{B_k}{2k} \mod \ZZ$ for all $k\ge 4$ (resp. $k\ge 2$).
	\item For every prime $p$ and $c\in \G$ there exists a $p$-adic measure $\mu_c\in\meas(\G,\Zp)$ such that
	\[
		(1-p^{2k-1})(1-c^{2k})b_{2k}=\int_{\G}{x^{2k}d\mu_c(x)} \text{ for all }\, k\geq 2 \,\text{(resp. $k\geq 1$)} .
	\]
	\end{enumerate}
\end{thm}

\begin{rem}
		Strictly speaking, the proof in \cite{andohopkinsrezk} is given only for the case of $\mathbb{E}_\infty$-String orientations, but the same proof also works in the case of $\mathbb{E}_\infty$-Spin orientations. 
\end{rem}

Our preparatory work on integer sequences of moments now immediately gives the following result.

\begin{thm}\label{thm:many_orient}
	There are bijections
		\[
		\begin{tikzcd}
			 \ZZ^\NNo\ar{r}{\cong} & \pi_0 \mathbb{E}_\infty (\MSpin,\KO)
		\end{tikzcd}
		\]
		and
		\[
		\begin{tikzcd}
			 \ZZ^\NNo\ar{r}{\cong} & \pi_0 \mathbb{E}_\infty (\MString,\KO).
		\end{tikzcd}
		\]
\end{thm}
\begin{proof}
	By \Cref{thm_ahs}, the set $\mathrm{C}_{\mathrm{string}}(gl_1(\KO))$ (resp. $\mathrm{C}_{\mathrm{spin}}(gl_1(\KO))$) is a torsor under the group $\Momeulzwei$ (resp. $\Momeuleins$) with an explicit trivialization given by:
	\begin{equation}\label{eq:c_is_mom}
		\begin{tikzcd}[row sep=tiny]
			\Momeulzwei \ar{r}{\simeq} & \mathrm{C}_{\mathrm{string}}(gl_1(\KO)), \\
			(\tilde{b}_k)_{k\geq 4} \ar[mapsto]{r} & (b_k:=\tilde{b}_k-\frac{B_k}{2k})_{k\geq 4}
		\end{tikzcd}
	\end{equation}
	and similarly for $\Momeuleins$ and $\mathrm{C}_{\mathrm{spin}}(gl_1(\KO))$. Now the claim follows from the isomorphisms
	\[
		\Phi_2 \colon \ZZ^\NNo \stackrel{\simeq}{\rightarrow} \Momeulzwei
	\]
	resp.
	\[
		\Phi_1 \colon \ZZ^\NNo \stackrel{\simeq}{\rightarrow} \Momeuleins
	\]
	in \Cref{thm3}.
\end{proof}

\subsection[Orientations of tmf]{$\mathbb{E}_\infty$-String orientations of tmf}\label{sec:string_of_tmf}

Recall the spectrum tmf of topological modular forms from \cite{tmf-book}.
We first recall the main result of \cite{andohopkinsrezk} about string orientations of tmf: For any subring $R\subseteq \CC$ let $\MF(R)=\bigoplus_{k\geq 0} \MF_k(R)$ denote the ring of modular forms with coefficients in $R$, see
\cite{deligne-rapoport}. For a prime $p$, let $\MFp$ be the ring of $p$-adic modular forms in the sense of Serre \cite{serre2}. The set $\mathrm{C}(gl_1(\tmf))$ is defined as the image of the characteristic map $b$, see \cite[Def. 5.11.]{andohopkinsrezk}:
\[
		\begin{tikzcd}[row sep=tiny]
			b\colon \pi_0 \mathbb{E}_\infty (\MString,\tmf) \ar{r} & \prod_{k\geq 4} \pi_{2k}(gl_1(\tmf)) \otimes \QQ.
		\end{tikzcd}	
\]

By \cite[Proposition 5.19]{andohopkinsrezk}, the characteristic map 
\[ b\colon \pi_0 \mathbb{E}_\infty(\MString,\tmf)\longrightarrow \mathrm{C}(gl_1(\tmf))\] is surjective with fibers principal
homogeneous under the group $A:=[bstring,gl_1(\tmf)]_{tors}$
of torsion classes in the group of maps of spectra $[bstring,gl_1(\tmf)]$. 
This group is a countably infinite
dimensional $\mathbb{F}_2$-vector space \cite[Remark 6.22]{hopkins_ICM}. The main theorem of \cite{andohopkinsrezk} describes  $\mathrm{C}(gl_1(\tmf))$ in terms of $p$-adic modular forms, and thus gives an arithmetic description of all $\mathbb{E}_\infty$-maps $\MString\to \tmf$ up to the ``torsion-ambiguity'' caused by the action of $A$.

\begin{thm}[{\cite[Thm. 12.1.]{andohopkinsrezk}}]\label{thm_ahr_tmf}
	The set  $\mathrm{C}(gl_1(\tmf))$ is the set of sequences $(g_k)_{k\geq 4}$ with $g_k\in \MF_k(\QQ),k\geq 4$ satisfying the following conditions:
	\begin{enumerate}
		\item We have $g_{2k+1}=0$ for $k\geq 2$.\footnote{This condition is vacuous because $M_{2k+1}(\QQ)=\{ 0\}$, but we include it to ease the comparison with \Cref{thm_ahs}.}
		\item For the $q$-expansion of $g_k$, we have $g_k(q)\equiv -\frac{B_{k}}{2k} \mod \ZZ\llbracket q\rrbracket$ for all $k\ge 4$.
		\item For each prime $p$ and each $c\in \G$ there exists a measure $\mu_c\in \meas(\G,\MF_p)$ such that for all $k\geq 4$:
		\[
				\left( \int_{\G}{x^{2k} d\mu_c(x)}\right)(q)=(1-c^{2k})(g_{2k}(q)-p^{2k-1}g_{2k}(q^p)).
		\]

		\item For all primes $p$ and all $k\ge 4$, we have $g_k|T(p)=(1+p^{k-1})g_k$, where $T(p)$ denotes the $p$-th Hecke operator.
	\end{enumerate}
\end{thm}

\begin{exmpl}\label{ex:eisenstein_works}
There is one sequence of modular forms well-known to satisfy the above conditions
\cite[Prop. 10.10.]{andohopkinsrezk}, namely the sequence of Eisenstein series
\[
	G_k(q)=-\frac{B_k}{2k}+\sum_{n\geq1}\sigma_{k-1}(n)q^n\,(k\ge 4\mbox{ even}),
\]
and $G_{2k+1}:=0$ for all $k\ge 2$. Here, $\sigma_{k-1}(n):=\sum\limits_{1\leq d\mid n}
d^{k-1}$ denotes the sum of the $(k-1)$st powers of all positive divisors of $n$.
\end{exmpl}

\begin{rem}\label{rem:multone}
	Condition $(d)$ in \Cref{thm_ahr_tmf} implies that each $g_k$ is a rational multiple of the Eisenstein series $G_k$, compare \cite{serre_arithmetics}[VII, 5.5, Proposition 13]. To see this, note that the Eisenstein series $G_k$ satisfies condition $(d)$, and that the common eigenspace of all $T(p)$ on $\MF_k(\QQ)$ is one dimensional by "multiplicity one", compare \cite[VII, 5.3, Corollary 1]{serre_arithmetics}.
\end{rem}

To understand the conditions of \Cref{thm_ahr_tmf}, we will need some preparation.

\begin{thm}[von Staudt-Clausen]\label{thm_vonstaudtclausen}
	Let $n\geq 2$ be even. Then
	\[
		B_{n}+\sum_{(p-1)|n}\frac{1}{p} \in \ZZ,
	\]
	where the sum extends over all primes $p>0$ such that $p-1$ divides $n$.
\end{thm}
\begin{proof}
	\cite[Theorem 5.10]{washington}.
\end{proof}

Recall that $v_p\colon\QQ^*\to\ZZ$ denotes the $p$-adic valuation, normalized by $v_p(p)=1$.

\begin{cor}\label{cor:valuation}
\begin{enumerate}
\item\label{estimation} For all primes $p$ and integers $k\ge 2$, we have
	\[ 
		v_p(B_k)\left\{\begin{array}{ccc} \ge 0 & ; & (p-1)\not |\,\, k\\
													=-1 & ; & (p-1)\mid k.\end{array}\right.
	\]												
\item\label{exact_value} For all primes $p$ and integers $r\ge 2$ we have
\[ v_p\left( \frac{B_{\varphi(p^r)}}{\varphi(p^r)}\right) = -r.\]
\item\label{integrality} For all primes $p$, integers $k\ge 1$ and $c\in\ZZ_p^*$, we have
	\[
			-(1-c^{2k})\frac{B_{2k}}{4k}\in\Zp .
 	\]
\end{enumerate}
\end{cor}

\begin{proof}
Part $(a)$ is immediate from \Cref{thm_vonstaudtclausen}. To see part $(b)$, we have
from \Cref{thm_vonstaudtclausen} that $B_{\varphi(p^r)}+a+1/p\in\mathbb{Z}_{(p)}$ for some 
$p$-adic unit $a$, hence $v_p(B_{\varphi(p^r)})=-1$, and the claim follows from $\varphi(p^r)=(p-1)p^{r-1}.$ To see part $(c)$, combining \cite[Thm. 10.4]{andohopkinsrezk} with \cite[Thm. 10.3, (b)]{andohopkinsrezk}, one sees that there is a $p$-adic measure on $\G$ with $2k$-th moment
$-\frac{B_{2k}}{4k}(1-c^{2k})(1-p^{2k-1})\in\ZZ_p$. As $-(1-p^{2k-1})$ is a $p$-adic unit, the claim follows.
\end{proof}

Recalling the subgroup $\Momzerozwei\subseteq
\ZZ^{\NN_{\ge 4}}$ from \Cref{def0}, we can now determine $\mathrm{C}(gl_1(\tmf))$ in terms of simultaneous congruences:

\begin{thm}\label{thm7}
	The map
	\[
		\begin{tikzcd}[row sep=tiny]
			\Psi_2\colon \Momzerozwei \ar{r} & \mathrm{C}(gl_1(\tmf))\\
			(q_k)_{k\geq 4} \ar[mapsto]{r} & (G_k+q_k G_k)_{k\geq 4}
		\end{tikzcd}
	\]
	is a bijection.
\end{thm}
\begin{proof}
	Let us first show that the map is well-defined, i.e. given $(q_k)_{k\geq 4}\in\Momzerozwei$, we check conditions $(a)$ through $(d)$ of \Cref{thm_ahr_tmf} for the sequence $g_k:=G_k+q_kG_k$.
 It is obvious that for $k\geq 2$ we have  
	\[ 
		g_{2k+1} = G_{2k+1}+q_{2k+1} G_{2k+1}=0\,\, (\mbox{condition }  (a)),
	\]
	and that for all primes $p$
	\[
		g_k|T(p) = (G_k+q_k G_k)|T(p)=(1+p^{k-1})(G_k+q_k G_k)=(1+p^{k-1})g_k
	\]
	holds, giving us condition $(d)$. 
	Condition $(c)$ is the assertion that for each prime $p$ and each $c\in \G$ there exists a unique measure $\mu_c\in \meas(\G,\MF_p)$ with $2k$th-moments
	\[
		\left( \int_\G x^{2k}\, d\mu_c(x)\right)(q) = (1+q_{2k})(1-c^{2k})(G_{2k}(q)-p^{2k-1}G_{2k}(q^p))\,\, (k\ge 4).
	\]
	
	Both
	\[
		((1-c^{2k})(G_{2k}(q)-p^{2k-1}G_{2k}(q^p)))_{k\ge 4} \quad \text{ and } (1+q_{2k})_{k\ge 4}
	\]
	are sequences of moments of $\MFp$-valued measures on $\G$: 
	The first one by \Cref{ex:eisenstein_works}, and the second one because
	$(q_k)_{k\ge 4}\in\Momzerozwei$. The convolution of these two measures is thus as
	required by condition $(c)$.
	
	It remains to show that $(g_k)_{k\ge 4}=(G_k+q_k G_k)_{k\geq 4}$ satisfies condition $(b)$ in \Cref{thm_ahr_tmf}. Since this is clear for odd $k$, it will suffice to show that
	\begin{equation}\label{eq:todo}
	-\frac{B_{2k}}{4k}q_{2k}\in\ZZ \text{ for all }k\geq 2.
	\end{equation}

This will imply (recalling that $q_k\in\ZZ$) that $q_{2k} G_{2k}(q)\in\ZZ\llbracket q\rrbracket$, and hence that 
	\[
			g_{2k}(q) = G_{2k}(q)+q_{2k} G_{2k}(q)\equiv -\frac{B_{2k}}{4k}  \mod \ZZ\llbracket q\rrbracket,
	\]
which is exactly condition $(b)$.

	To establish \eqref{eq:todo}, let $p$ be a prime. By the definition of Mom$^{(0)}_{\ge 4}$ there exists a measure $\mu_p\in \meas(\G,\Zp)$ with $2k$-th moments $(q_{2k})_{k\ge 2}$ and total mass
	\[
		\int_{\G}1\,\, d\mu_p=0.
	\]
	It follows from \Cref{prop_regularize} that there exists a measure $\tilde{\mu}_p\in\meas(\G,\Zp)$ with $\mu_p=(1-c_*)\tilde{\mu}_p$ where $c$ is a fixed topological generator of $\G$. This implies:
	\[
		-\frac{B_{2k}}{4k}q_{2k} = -\frac{B_{2k}}{4k} \int_{\G} x^{2k} d\mu_p =-\underbrace{\vphantom{\int_{\G} x^{2k} d\tilde{\mu}_p}\frac{B_{2k}}{4k} (1-c^{2k})}_{\in \Zp} 
		\underbrace{\int_{\G} x^{2k} d\tilde{\mu}_p }_{\in \Zp},
	\]
	see \Cref{cor:valuation}, \ref{integrality} for the first containment.
	Since $-\frac{B_{2k}}{4k}q_{2k} \in \Zp$ holds for all primes $p$,
	we have shown \eqref{eq:todo}, and hence that the map $\Psi_2$ is well-defined.\par
	The injectivity of $\Psi_2$ is clear since the coefficient of $q$ in the $q$-expansion of $G_k+q_k G_k$ is $1+q_k$. To see the surjectivity, recall that condition $(d)$ in \Cref{thm_ahr_tmf} implies that every $(g_k)_{k\geq 4}\in \mathrm{C}(gl_1(\tmf))$ is of the form $g_k=G_k+q_k G_k$ for some sequence $(q_k)_{k\geq 4}\in \QQ^\NN$, which we can assume to be even by condition $(a)$ (see also \Cref{rem:multone}). Noting that the coefficient of $q$ in the $q$-expansion of $q_k G_k$ is $q_k$ implies  $q_k\in\ZZ$ for all $k\geq 4$. 
We need to see that in fact $(q_k)_{k\ge 4}\in\Momzerozwei\subseteq\ZZ^{\NN_{\ge 4}}$. For all $c\in\G$, looking at the linear term in condition $(c)$ of \Cref{thm_ahr_tmf}, gives us a measure $\mu_{p,c}\in\meas(\G,\Zp)$ such that
	\[
			\int_{\G}x^{2k}d\mu_{p,c}(x) =(1-c^{2k})(1+q_{2k})\,\,(k\ge 4).
	\]
	From the equivalence of condition \Bp\ and \Bpa\, ,  compare \Cref{cor_BpBpa}, and subtracting the constant measure with mass one, there also exists a (unique) measure $\mu_p\in\meas(\G,\Zp)$ such that
	\begin{equation}\label{eq:moments}
			\int_{\G}x^{2k}d\mu_{p}(x) =q_{2k}\,\, (k\ge 4).
	\end{equation}
	In order to show $(q_k)_{k\geq 4}\in\Momzerozwei$, it remains to show that
	\[
		\int_{\G}1\,d\mu_{p} =0
	\]
	for all primes $p$. 
	Condition $(b)$ in \Cref{thm_ahr_tmf} for $g_k=G_k+q_kG_k$ implies $-\frac{B_k}{2k}q_k\in\ZZ$ ($k\ge 4$). Choosing $k=\varphi(p^r)$ for some $r\ge 2$, this implies 
	\begin{equation}\label{eq:pole}
		v_p(q_{\varphi(p^r)})\geq -v_p\left(-\frac{B_{\varphi(p^r)}}{2\varphi(p^r)}\right)
		\stackrel{\text{Cor.}\ref{cor:valuation},\ref{exact_value}}{\ge} r.
	\end{equation}
We can now compute
	\[
		\int_{\G}1\,d\mu_{p}(x) = \lim_{r\rightarrow \infty}\int_{\G} x^{\varphi(p^r)}d\mu_{p}(x)\stackrel{\eqref{eq:moments}}{=} \lim_{r\rightarrow \infty} q_{\varphi(p^r)} \stackrel{\eqref{eq:pole}}{=}0,
	\]
 as desired. This concludes the proof of the surjectivity of $\Psi_2$.
\end{proof}

\begin{cor}
	There is a bijection
	\[
		\widehat{\ZZ}\times \ZZ^\NN \simeq \pi_0 \mathbb{E}_\infty (\MString,\tmf)/[bstring,gl_1(\tmf)]_{tors}.
	\]
\end{cor}

\begin{proof}
 Follows immediately from \Cref{thm7} and \Cref{thm5}, along with the discussion at the beginning of the present subsection.
\end{proof}

\subsection[Evaluation at the cusp]{Evaluating $\mathbb{E}_\infty$-String orientations of tmf at the cusp}\label{sec:eval_at_cusp}

Evaluating a modular form at the cusp refines to an 
$\mathbb{E}_\infty$-map\newline
 $\mathrm{ev}\colon\tmf\to\KO$ (\cite[Thm. A8]{hill_lawson}), and we ask about the
induced map on $\mathbb{E}_\infty$-String orientations
 \[ \mathrm{ev}_* \colon\pi_0 \mathbb{E}_\infty(\MString,\tmf)/[bstring,gl_1(\tmf)]_{\mathrm{tors}}\longrightarrow \mathbb{E}_\infty(\MString,\KO). \footnote{Recall that there is no torsion-ambiguity for maps to $\KO$, hence
 ev$_*$ factors as indicated.}\]
 Using \Cref{thm_ahr_tmf} and \Cref{thm7} for the source and \Cref{thm_ahs} and
 \eqref{eq:c_is_mom} in the proof of \Cref{thm:many_orient} for the target of ev$_*$, this map is identified with a map
 \[ \mathrm{ev}\colon \Momzerozwei\longrightarrow\Momeulzwei,\]
 and we leave to the reader to check that this map is given by
 \begin{equation}\label{eq:ev}
  \mathrm{ev}\left( (q_k)_{k\ge 4}\right) = (-\frac{B_k}{2k}q_k)_{k\ge 4}.
  \end{equation}
 Since $B_k\neq 0$ for even integers $k$, for example because $v_2(B_k)=-1$ by \Cref{cor:valuation},\ref{estimation}, the map $\mathrm{ev}_*$ is injective. To describe its image, we
 first need a short reminder on some classical number theory.\par

Recall from \cite[Thm 10.6]{andohopkinsrezk} and \cite[Thm 10.3, (b)]{andohopkinsrezk} that for every prime $p$ and $c\in\ZZ_p^*$ there
exist a (unique) measure $\mu^{\mathrm{zeta}}_{p,c}\in\meas(\G,\ZZ_p)$
with moments
\begin{equation}\label{eq:zeta_moments}
 \int_\G x^{2k}\,d\mu^{\mathrm{zeta}}_{p,c}(x) = -(1-c^{2k})(1-p^{2k-1})\frac{B_{2k}}{4k}\, , \, k\ge 1.
\end{equation}
Recalling that $\meas(\G,\ZZ_p)$ is a commutative algebra under convolution,
we define {\em the zeta-ideal} to be the principal ideal 
\[ \mathfrak{Z}_p\subseteq \meas(\G,\ZZ_p)\]
generated by $\mu^{\mathrm{zeta}}_{p,c}$ for any choice of topological
generator $c\in\G$. This well-defined by the following result.

\begin{lem}\label{lem_zetaideal}
	The ideal generated by $\mu_{p,c}^{\mathrm{zeta}}$ is independent of $c$. 
\end{lem}
\begin{proof}
We note that the convolution $*$ on $\meas(\G,\ZZ_p)$ is such that for every
$c\in\G$ and $\mu\in\meas(\G,\ZZ_p)$ we have $c_*(\mu)=c_*(1)*\mu$.
Now fix two topological generators $c,\tilde{c}\in\G$. Then \Cref{prop_regularize}
implies that $1-c_*(1)$ and $1-\tilde{c}_*(1)$ generate the same ideal of
$\meas(\G,\ZZ_p)$ (namely the kernel of $\mu\mapsto\int_\G d\mu$). Hence there is a unit $u$ such that $1-c_*(1) = u * (1-\tilde{c}_*(1))$. A straight-forward comparison of moments shows that $\mu_{p,c}^{\mathrm{zeta}}=u * \mu_{p,\tilde{c}}^{\mathrm{zeta}}$, as desired.
\end{proof}

We now determine the image of 
 \[ \mathrm{ev}\colon \Momzerozwei\hookrightarrow\Momeulzwei,\]
 as follows. Given $(b_k)_{k\ge 4}\in\Momeulzwei$, a prime $p$
 and a topological generator $c\in\G$ there is, by definition of 
 $\Momeulzwei$, a measure $\mu_{p,c}\in\meas(\G,\ZZ_p)$ with $2k$-th
 moments equal to 
 \begin{equation}\label{eq:moments_of_b}
 \int_\G x^{2k}\,d\mu_{p,c}(x) = (1-c^{2k})(1-p^{2k-1})b_{2k} \,\, (k\ge 2).
 \end{equation}
 
 \begin{thm}\label{thm:zeta_ideal}
 In the above situation, the following are equivalent:
 \begin{enumerate}
 \item The sequence $(b_k)_{k\ge 4}$ is in the image of ev.
 \item For every prime $p$, $\mu_{p,c}\in \mathfrak{Z}_p$.
 \end{enumerate}
 \end{thm}
 
 \begin{proof}
 Assume $(a)$. By \Cref{eq:ev} we have $b_k=-\frac{B_k}{2k}q_k$
 for a sequence $(q_k)\in\Momzerozwei$. A comparison of moments using 
 \Cref{eq:zeta_moments} and \Cref{eq:moments_of_b} shows that $\mu_{p,c}$
 is the convolution of $\mu_{p,c}^{\mathrm{zeta}}$ and the measure 
 corresponding to $(q_k)$, hence showing $(b)$.\par
Now assume $(b)$ and write $\mu_{p,c}=\mu_{p,c}^{\mathrm{zeta}}*\nu_p$
for a suitable $\nu_p\in\meas(\G,\ZZ_p)$. On moments, this equality yields
after canceling factors $(1-c^{2k})(1-p^{2k-1})$
\begin{equation}\label{eq:relation}
 b_{2k} = -\frac{B_{2k}}{4k}\beta_{2k}\mbox{ in }\ZZ_p\mbox{ for all }k\ge 2,
 \end{equation}
with $\beta_k$ denoting the $k$-th moment of $\nu_p$.
This equation firstly shows that for all $p$ and $k$, we have $\beta_{2k}
\in\QQ\cap\ZZ_p$, and hence that $\beta_{2k}\in\ZZ$. To conclude that 
$(\beta_k)\in\Momzerozwei$ (and hence $(b_k)$ is in the image of ev
by \Cref{eq:ev}), it remains to see that 
$\int_\G 1 \,d\nu_p(x)=0$ for all primes $p$. This follows from
\Cref{eq:relation} as in the proof of \Cref{thm7}.
 \end{proof}

\begin{rem}
	We can express the condition $(b)$ more explicitly in terms of moments as follows: A sequence $(b_k)_{k\ge 4}\in\Momeulzwei$ is in the image of the \emph{evaluation at the cusp map} \text{ev} if and only if $\beta_{2k}:=-4k\frac{b_{2k}}{B_{2k}}$, $k\geq 2$ is a sequence of integers which is contained in
	\[
		(\beta_k)_{k\geq 4} \in \Momzerozwei.
	\]
	Checking this last condition is tantamount to check the corresponding generalized Kummer congruences which is in general a difficult task. Nevertheless, we can derive the following necessary condition for $(b_k)_{k\ge 4}\in\Momeulzwei$ being contained in the image of \text{ev}: Let us write
	\[
		B_{2k}=\frac{n_{2k}}{d_{2k}},\quad \text{gcd}(n_{2k},d_{2k})=1
	\]
	with $n_{2k}\in\ZZ_{\geq 0}$ the numerator and $d_{2k}\in\ZZ$ the denominator of the Bernoulli numbers. A sequence $(b_k)_{k\ge 4}\in\im(\text{ev})$ is always divisible by the numerators of the Bernoulli numbers, i.e. $n_{2k}|b_{2k}$ for all $k\geq 2$.	This gives an obstruction against being contained in the image of the evaluation map which is easier to check in practice. Using this obstruction we will see in the following example that the map $\text{ev}_*$ is not surjective:
\end{rem}

\begin{exmpl}
The map 
 \[ \mathrm{ev}_* \colon\pi_0 \mathbb{E}_\infty(\MString,\tmf)/[bstring,gl_1(\tmf)]_{\mathrm{tors}}\longrightarrow \mathbb{E}_\infty(\MString,\KO)\]
 is not surjective. To see this, we have to exhibit
 a sequence $(b_k)\in\Momeulzwei$ not in the image of ev.
 We can do this explicitly using our isomorphism $\Phi_2$ from \Cref{thm3}, as follows.
  Let $a$ be an integer which is prime to $691$. Note that the prime $691$ is the numerator of  $\frac{B_{12}}{2\cdot 12}$. Let $l\in\ZZ^{\NNo}$ be any sequence starting with $l=(0,0,0,0,a,...)$. This gives $\Phi_2(l)=(b_4,0,b_6,0,b_8,0,b_{10},0,b_{12},0,...)\in\Momeulzwei $ with $b_4=b_6=b_8=b_{10}=0$ and $b_{12}=a\cdot\Phi_{8,4}^{(2)}$. From \Cref{eq:defPhi2} preceding \Cref{thm3}, we see that 
  $v_p(\Phi_{8,4}^{(2)})=0$ for all primes $p\not\in S_4=\{ p\mid  p-1 \leq 8\}$.
  In particular, $v_{691}(\Phi_{8,4}^{(2)})=0$ and
   thus $\Phi_2(l)$ is not in the image of ev since $v_{691}(b_{12})=0$ but every element $(a_k)_{k\geq4}\in\im(\mathrm{ev})$ is of the form $a_k=-\frac{B_k}{2k}q_k$ for some $(q_k)_{k\geq 4}\in\Momzerozwei$ and thus satisfies
	\[
		v_{691}(a_{12})=v_{691}\left(-\frac{B_{12}}{2\cdot 12}\right)+v_{691}(q_{12})=1+v_{691}(q_{12})\geq 1.
	\]
	The same argument works with any irregular prime $p$ in place of $691$.
\end{exmpl}

\subsection{Comparing $\mathbb{E}_\infty$-String and $\mathbb{E}_\infty$-Spin orientations of KO}\label{sec:comparing}

\Cref{thm:many_orient} implies in particular that there are uncountably many
$\mathbb{E}_\infty$-String orientations $\MString\to\KO$ besides the composition
$\MString\xrightarrow{\pi}\MSpin\xrightarrow{\hat{A}}\KO$, where $\pi$ denotes the 
canonical map. To understand these orientations a bit more, we wish to ask which of them 
come from an $\mathbb{E}_\infty$-Spin orientation, i.e. we want to know about (the image of) the map
\[ \pi^*\colon\pi_0\mathbb{E}_\infty(\MSpin,\KO)\longrightarrow\pi_0\mathbb{E}_\infty(\MString,\KO).\]
We find the following.

\begin{thm}\label{thm:stringspinko}
There is a Cartesian square of sets
\[
\begin{tikzcd}[row sep=large]
	\pi_0\mathbb{E}_\infty(\MSpin,\KO) \ar[d,two heads] \ar[r,"{\pi^*}"] &  \pi_0\mathbb{E}_\infty(\MString,\KO) \ar[d,"{\alpha}"] \\
	\ZZ \ar[r,hook] & \widehat{\ZZ}
\end{tikzcd}
\]
with injections and surjections as indicated. We have $\alpha(\pi^*(\widehat{A}))=0$.
\end{thm}

Here, $\ZZ\hookrightarrow\widehat{\ZZ}$ is the inclusion of $\ZZ$ into its pro-finite
completion. Put differently, for every $\mathbb{E}_\infty$-String orientation $f:\MString\to\KO$
there is an invariant $\alpha(f)\in\widehat{\ZZ}$ such that $f$ factors (necessarily uniquely)
through $\MSpin$ if and only if $\alpha(f)$ lies in the dense subgroup $\ZZ\subseteq
\widehat{\ZZ}$. Since $\widehat{\ZZ}/\ZZ$ is uncountable, there are many $f$'s which 
do not factor over $\MSpin$.

\begin{proof}[Proof of \Cref{thm:stringspinko}]
Using \Cref{thm_ahs} and \Cref{eq:c_is_mom} in the proof of \Cref{thm:many_orient}, in the following
we identify the map $\pi^*$ with the map
\[ \mathrm{Mom}^{\mathrm{Euler}}_{\ge 2}\longrightarrow\mathrm{Mom}^{\mathrm{Euler}}_{\ge 4}\, , \, (b_k)_{k\ge 2}\mapsto (b_k)_{k\ge 4}. \]
We then define $\alpha\colon\mathrm{Mom}^{\mathrm{Euler}}_{\ge 4}\to\widehat{\ZZ}
\simeq\prod\limits_p\ZZ_p$ by setting 
\[ \left( \alpha((b_k)_{k\ge 4})\right)_p := \int_\G x^2\, d\mu_p(x)\, \in\ZZ_p\]
for all primes $p$, where $d\mu_p\in\meas(\G,\ZZ_p)$ is the unique measure with
moments $b_k$ ($k\ge 4$). Since any tail of moments determines the corresponding measure, $\pi^*$
is injective. It is clear that a sequence $(b_k)_{k\ge 4}\in\mathrm{Mom}^{\mathrm{Euler}}_{\ge 4}$
lies in $\mathrm{Mom}^{\mathrm{Euler}}_{\ge 2}$ if and only if 
$\alpha((b_k)_{k\ge 4})\in\ZZ$. Using the variant of \Cref{thm2} for $m=2$
but with $b_2\in\widehat{\ZZ}$ (rather than $b_2\in\ZZ$) being allowed, one checks that $\alpha$ is surjective. Finally, the fact that $\alpha(\pi^*(\widehat{A}))=0$
follows because we used $\widehat{A}$ to trivialize the torsor in \Cref{eq:c_is_mom}.
\end{proof}

\bibliographystyle{amsalpha} 
\bibliography{interpolation_and_orientations}
\end{document}